\documentclass[11pt]{amsart}
\usepackage{amscd,amssymb,amsthm,amsmath,amssymb,mathrsfs,enumerate,epigraph}
\usepackage[matrix,arrow,curve]{xy}
\usepackage[margin=2cm]{geometry}
\usepackage{mathtools}

\newtheorem{theorem}{Theorem}[section]

\newtheorem{lemma}[theorem]{Lemma}
\newtheorem*{lemma*}{Lemma}

\newtheorem*{corollary*}{Corollary\,A}
\newtheorem*{theorem*}{Theorem}
\newtheorem*{mtheorem*}{Main Theorem}
\newtheorem*{proposition*}{Proposition\,B}

\theoremstyle{definition}
\newtheorem{example}[theorem]{Example}
\newtheorem*{example*}{Example}

\theoremstyle{remark}

\makeatletter\@addtoreset{equation}{section} \makeatother

\title{Smooth Fano 3-folds satisfying Condition $(\mathbf{A})$}

\author{Hamid Abban, Ivan Cheltsov, Takashi Kishimoto, Fr\'ed\'eric Mangolte}

\let\origmaketitle\maketitle
\def\maketitle{
  \begingroup
  \def\uppercasenonmath##1{} 
  \let\MakeUppercase\relax 
  \origmaketitle
  \endgroup
}

\dedicatory{To Yuri Tschinkel on the~occasion of his 60th birthday.}

\begin{document}

\begin{abstract}
A smooth variety is said to satisfy Condition $(\mathbf{A})$ if every finite abelian subgroup of its automorphism group has a fixed point.
We classify smooth Fano 3-folds that satisfy Condition $(\mathbf{A})$.
\end{abstract}

\address{\emph{Hamid Abban}\newline
\textnormal{University of Nottingham, Nottingham, England
\newline
\texttt{hamid.abban@nottingham.ac.uk}}}

\address{ \emph{Ivan Cheltsov}\newline
\textnormal{University of Edinburgh, Edinburgh, Scotland
\newline
\texttt{i.cheltsov@ed.ac.uk}}}

\address{\emph{Takashi Kishimoto}\newline
\textnormal{Saitama University, Saitama, Japan
\newline
\texttt{kisimoto.takasi@gmail.com}}}

\address{\emph{Fr\'ed\'eric Mangolte}\newline
\textnormal{Aix Marseille University, CNRS, I2M, Marseille, France
\newline
\texttt{frederic.mangolte@univ-amu.fr}}}

\maketitle


For simplicity of exposition, all varieties are assumed to be projective, normal, irreducible, 
and defined over $\mathbb{C}$, the~field of complex numbers, or sometimes a subfield $\Bbbk\subseteq\mathbb{C}$, unless otherwise stated explicitly. 

\section{Introduction}
\label{sec:introduction}

Fano varieties play a central role in algebraic geometry, particularly in the classification of projective varieties and the study of their birational properties. Among these, smooth Fano 3-folds over $\mathbb{C}$ have been extensively studied, with their deformation families systematically classified by Iskovskhikh, Mori, and Mukai. A natural question in the study of such varieties concerns the structure of their automorphism groups and the existence of fixed points under group actions. In this context, we pay particular attention to Condition~$\mathbf{(A)}$, which stipulates that every finite abelian subgroup of the automorphism group of a smooth variety fixes at least one point.

This paper investigates which smooth Fano 3-folds satisfy Condition~$\mathbf{(A)}$, building on prior work in \cite{AbbanCheltsovKishimotoMangolte-1,AbbanCheltsovKishimotoMangolte-2}. Our main result, presented in the Main Theorem below, provides a complete classification of the $105$ deformation families of smooth Fano 3-folds with respect to Condition~$\mathbf{(A)}$. Specifically, we identify families where all members satisfy Condition $\mathbf{(A)}$, families where no members satisfy it, and families containing members that do not satisfy it. This classification leverages the Mori-Mukai notation \cite{fanography,IsPr99} and relies on detailed analyses of automorphism groups and their actions.

As a consequence of our Main Theorem and results from \cite{AbbanCheltsovKishimotoMangolte-2}, we derive a corollary concerning the unirationality of Fano 3-folds over subfields $\Bbbk\subseteq\mathbb{C}$. We further explore the existence of rational points, establishing in Proposition\,B below that smooth Fano 3-folds in certain families always admit $\Bbbk$-points, while others contain members with no $\Bbbk$-points for specific subfields, such as $\mathbb{R}$ or $\mathbb{Q}$. These findings connect the geometric properties of Fano 3-folds to arithmetic questions, shedding light on their behavior over non-closed fields. The main result of this paper is the following theorem.

\begin{mtheorem*}
Let $X$ be a smooth Fano 3-fold. If $X$ is contained in one of the deformation families
\begin{center}
\textnumero 1.10, \textnumero 1.11, \textnumero 1.15, \textnumero 2.1, \textnumero 2.9, \textnumero 2.11, \textnumero 2.13, \textnumero 2.14, \textnumero 2.15,\\ \textnumero 2.17, \textnumero 2.20, \textnumero 2.22,  \textnumero 2.26, \textnumero 2.28, \textnumero 2.30, \textnumero 2.31, \textnumero 2.35,\\ 
\textnumero 2.36,  \textnumero 3.8, \textnumero 3.11, \textnumero 3.14, \textnumero 3.15, \textnumero 3.16, \textnumero 3.18, \textnumero 3.21, \textnumero 3.22,\\ \textnumero 3.23, \textnumero 3.24, \textnumero 3.26, \textnumero 3.29, \textnumero 3.30, \textnumero 4.5, \textnumero 4.9, \textnumero 4.11, \textnumero 5.1,
\end{center}
then $X$ satisfies Condition~$\mathbf{(A)}$. If $X$ is contained in one of the deformation families
\begin{center}
\textnumero 1.14, \textnumero 1.16, \textnumero 1.17, \textnumero 2.25, \textnumero 2.27, \textnumero 2.29, \textnumero 2.32, \textnumero 2.33, \textnumero 2.34, \textnumero 3.17, \\
\textnumero 3.19, \textnumero 3.20, \textnumero 3.25, \textnumero 3.27, \textnumero 3.28, \textnumero 3.31,  \textnumero 4.1, \textnumero 4.2, \textnumero 4.3, \textnumero 4.4,  \textnumero 4.6, \\
\textnumero 4.7, \textnumero 4.8, \textnumero 4.10,  \textnumero 4.12, \textnumero 5.2, \textnumero 5.3, \textnumero 6.1, \textnumero 7.1, \textnumero 8.1, \textnumero 9.1, \textnumero 10.1,
\end{center}
then $X$ does not satisfy Condition~$\mathbf{(A)}$. Finally, every (remaining) deformation family
\begin{center}
\textnumero 1.1, \textnumero 1.2, \textnumero 1.3, \textnumero 1.4, \textnumero 1.5, \textnumero 1.6, \textnumero 1.7, \textnumero 1.8, \textnumero 1.9, \textnumero 1.12, \\
\textnumero 1.13, \textnumero 2.2, \textnumero 2.3, \textnumero 2.4, \textnumero 2.5, \textnumero 2.6,  \textnumero 2.7, \textnumero 2.8, \textnumero 2.10, \textnumero 2.12, \\
\textnumero 2.16, \textnumero 2.18, \textnumero 2.19, \textnumero 2.21, \textnumero 2.23, \textnumero 2.24,  \textnumero 3.1, \textnumero 3.2,  \textnumero 3.3,  \textnumero 3.4,  \\ \textnumero 3.5, \textnumero 3.6, \textnumero 3.7, \textnumero 3.9, \textnumero 3.10, \textnumero 3.12, \textnumero 3.13, \textnumero 4.13,
\end{center}
contains a smooth Fano 3-fold that does not satisfy Condition~$\mathbf{(A)}$.
\end{mtheorem*}

\begin{corollary*}[c.f.\ {\cite{CheltsovTschinkelZhang2025},\cite[Theorem 1.1]{DuncanReichstein}}]
Every deformation family among
\begin{center}
\textnumero 1.1, \textnumero 1.2, \textnumero 1.3, \textnumero 1.4, \textnumero 1.5, \textnumero 1.6, \textnumero 1.7, \textnumero 1.8, \textnumero 1.9, \textnumero 1.12, \\
\textnumero 1.13, \textnumero 1.14, \textnumero 1.16, \textnumero 1.17,  \textnumero 2.2, \textnumero 2.3, \textnumero 2.4, \textnumero 2.5, \textnumero 2.6,  \textnumero 2.7, \\
\textnumero 2.8, \textnumero 2.10, \textnumero 2.12, \textnumero 2.16, \textnumero 2.18, \textnumero 2.19, \textnumero 2.21, \textnumero 2.23, \textnumero 2.24, \textnumero 2.25, \\
\textnumero 2.27, \textnumero 2.29, \textnumero 2.32, \textnumero 2.33, \textnumero 2.34, \textnumero 3.1, \textnumero 3.2,  \textnumero 3.3,  \textnumero 3.4,  \textnumero 3.5, \\
\textnumero 3.6, \textnumero 3.7, \textnumero 3.9, \textnumero 3.10, \textnumero 3.12, \textnumero 3.13, \textnumero 3.17, \textnumero 3.19, \textnumero 3.20, \textnumero 3.25, \\
\textnumero 3.27, \textnumero 3.28, \textnumero 3.31,  \textnumero 4.1, \textnumero 4.2, \textnumero 4.3, \textnumero 4.4,  \textnumero 4.6, \textnumero 4.7, \textnumero 4.8, \\
\textnumero 4.10,  \textnumero 4.12,  \textnumero 4.13, \textnumero 5.2, \textnumero 5.3, \textnumero 6.1, \textnumero 7.1, \textnumero 8.1, \textnumero 9.1, \textnumero 10.1
\end{center}
contains a smooth Fano 3-fold defined over some subfield $\Bbbk\subseteq\mathbb{C}$ that is not $\Bbbk$-unirational.
\end{corollary*}
    
We expect that every family listed in this corollary contains a smooth Fano 3-fold $X$ defined over some subfield $\Bbbk\subseteq\mathbb{C}$ such that $X(\Bbbk)=\varnothing$.
For deformation families 
\begin{center}
\textnumero 1.6, \textnumero 1.8, \textnumero 1.9, \textnumero 1.13, \textnumero 1.14,
\textnumero 1.16, \textnumero 1.17, \textnumero 2.12, \textnumero 2.21, \textnumero 2.32, \textnumero 3.27, \textnumero 4.1,
\end{center}
this expectation follows from the corollary and \cite{Kollar,KuznetsovProkhorov2023,KuznetsovProkhorov2024}.
This is also true for the remaining deformation families by the following more precise result.

\begin{proposition*}
Let $X$ be a smooth Fano 3-fold defined over a subfield $\Bbbk\subseteq\mathbb{C}$. Suppose that $X$ is contained in one of the following deformation families:
\begin{center}
\textnumero 1.11, \textnumero 1.15, \textnumero 2.1, \textnumero 2.9, \textnumero 2.11, \textnumero 2.14, \textnumero 2.15, \textnumero 2.17, \textnumero 2.20, \textnumero 2.22, \textnumero 2.26, \\
\textnumero 2.28, \textnumero 2.30, \textnumero 2.31, \textnumero 2.35, \textnumero 2.36, \textnumero 3.8, \textnumero 3.11, \textnumero 3.14, \textnumero 3.15, \textnumero 3.16,  \textnumero 3.18, \\
\textnumero 3.21, \textnumero 3.22, \textnumero 3.23, \textnumero 3.24, \textnumero 3.26, \textnumero 3.29, \textnumero 3.30, \textnumero 4.5, \textnumero 4.9, \textnumero 4.11, \textnumero 5.1.
\end{center}
Then $X(\Bbbk)\ne\varnothing$. Moreover, every deformation family among
\begin{center}
\textnumero 1.1, \textnumero 1.2, \textnumero 1.3, \textnumero 1.4, \textnumero 1.5, \textnumero 1.10, \textnumero 1.12, \textnumero 1.14, \textnumero 1.16,
\textnumero 1.17, \textnumero 2.2, \textnumero 2.3, 
\textnumero 2.4, \\ 
\textnumero 2.6, \textnumero 2.7, \textnumero 2.8, \textnumero 2.10, \textnumero 2.12,
\textnumero 2.13, \textnumero 2.16, \textnumero 2.18, \textnumero 2.19, \textnumero 2.21, \textnumero 2.23, 
\textnumero 2.25, \textnumero 2.27, \\ 
\textnumero 2.29,
\textnumero 2.32, \textnumero 2.33,  \textnumero 2.34, \textnumero 3.1, \textnumero 3.2, \textnumero 3.3, \textnumero 3.4, \textnumero 3.5,  \textnumero 3.6, 
\textnumero 3.9,
\textnumero 3.10, \textnumero 3.12, \\  
\textnumero 3.13, \textnumero 3.17, \textnumero 3.19, \textnumero 3.20, \textnumero 3.25, \textnumero 3.27, \textnumero 3.28,
\textnumero 3.31, \textnumero 4.1, 
\textnumero 4.2, \textnumero 4.3, 
\textnumero 4.4, \\ 
\textnumero 4.6, \textnumero 4.7, \textnumero 4.8, \textnumero 4.10, \textnumero 4.12,  \textnumero 4.13, \textnumero 5.2, \textnumero 5.3, \textnumero 6.1, \textnumero 7.1, \textnumero 8.1, \textnumero 9.1, \textnumero 10.1
\end{center}
contains a real smooth pointless Fano 3-fold, each  family 
\textnumero 1.9, \textnumero 1.13,  \textnumero 2.5,  \textnumero 2.24, \textnumero 3.7
contains a smooth Fano 3-fold defined over $\mathbb{Q}$ that does not have rational points,
and families \textnumero 1.6, \textnumero 1.7, \textnumero 1.8 contain smooth members defined over 
a subfield $\Bbbk\subseteq\mathbb{C}$ which have no $\Bbbk$-points. 
\end{proposition*}

The paper is organized as follows. Section~\ref{section:main} contains the proof of the Main Theorem, with a detailed case-by-case analyses of deformation families. In Section~\ref{section:auxiliary}, we prove Proposition\,B. Finally, Appendix \ref{section:V14} addresses the unirationality of degree $14$ smooth Fano 3-folds in family \textnumero 1.7, providing a proof that such a 3-fold is $\Bbbk$-unirational if and only if it has a $\Bbbk$-point. Since this paper is the threequel of \cite{AbbanCheltsovKishimotoMangolte-1,AbbanCheltsovKishimotoMangolte-2}, we will use many results obtained in these two papers. 

\medskip
\noindent
\textbf{Acknowledgements.}
We thank Paolo Cascini and Yuri Prokhorov for very fruitful discussions, and we thank Zhijia Zhang for the help with checking the smoothness of the 3-fold constructed in Example~\ref{example:A-1-8}.
Hamid Abban has been supported by the EPSRC grant EP/Y033450/1 and the Royal Society International Collaboration Award ICA$\backslash$1$\backslash$231019.
Ivan Cheltsov has been supported by the Simons Collaboration grant \emph{Moduli of Varieties}. Takashi Kishimoto has been supported by the JSPS KAKENHI Grant Number 23K03047. 
We also thank CIRM, Luminy, for the hospitality provided during a semester-long \emph{Morlet Chair} and for creating a perfect work environment.

\section{Proof of the Main Theorem}
\label{section:main}

Let $X$ be a smooth Fano 3-fold. If $X$ is contained in one of the deformation families
\begin{center}
\textnumero 1.10, \textnumero 1.15, \textnumero 2.9, \textnumero 2.11, \textnumero 2.13, \textnumero 2.14, \textnumero 2.17, \textnumero 2.20, \textnumero 2.22, \textnumero 2.26, \\
\textnumero 2.28, \textnumero 2.30, \textnumero 2.31, \textnumero 2.35, \textnumero 2.36,  \textnumero 3.8, \textnumero 3.11, \textnumero 3.14, \textnumero 3.16, \textnumero 3.18, \\
\textnumero 3.21, \textnumero 3.22, \textnumero 3.23, \textnumero 3.24, \textnumero 3.26, \textnumero 3.29, \textnumero 3.30, \textnumero 4.5, \textnumero 4.9, \textnumero 4.11,
\end{center}
then it follows from \cite{AbbanCheltsovKishimotoMangolte-2} that $X$ satisfies Condition~$\mathbf{(A)}$. 

\begin{lemma}
\label{lemma:A}
Let $A$ be a finite abelian subgroup of the group $\mathrm{Aut}(X)$. 
Suppose that $X$ is contained in one of the~following families: \textnumero 1.11, \textnumero 2.1, \textnumero 2.15, \textnumero 3.15, \textnumero 5.1.
Then $A$ fixes a point in $X$. 
\end{lemma}

\begin{proof}
If $X$ is contained in the family \textnumero 1.11, then $-K_X\sim 2H$ for an ample divisor $H\in\mathrm{Pic}(X)$ with $H^3=1$,
and the base locus of the linear system $|H|$ consists of a single point which must be fixed by the entire automorphism group $\mathrm{Aut}(X)$. In particular, it is fixed by $A$.
Similarly, if $X$ is contained in the deformation family \textnumero 2.1,
then it follows from \cite[Lemma~2.13]{AbbanCheltsovKishimotoMangolte-2} that there exists an $A$-equivariant birational morphism $X\to Y$, where $Y$ is a smooth Fano 3-fold in the family \textnumero 1.11, so $A$ fixes a point in $Y$ and, therefore, it follows from \cite[Proposition~A.4]{ReYou00} that $A$ fixes a point in $X$.

Suppose that $X$ is contained in the family \textnumero 3.15.
Then $X$ is the blowup of a smooth quadric $Q\subset\mathbb{P}^4$ along a disjoint union of a line $\ell$ and a smooth conic,
and it follows from \cite{Matsuki} that this blowup is $A$-equivariant and the line $\ell$ is $A$-invariant, hence Duncan's lemma \cite[Lemma~2.4]{AbbanCheltsovKishimotoMangolte-2} and \cite[Corollary 2.5]{AbbanCheltsovKishimotoMangolte-2} imply that the group $A$ fixes a point in $\ell$.
Therefore, $A$ fixes a point in $X$ by \cite[Proposition~A.4]{ReYou00}.

Next, we assume that $X$ is contained in the family \textnumero 2.15. Then it follows from \cite[Lemma~2.13]{AbbanCheltsovKishimotoMangolte-2} that there exists an $A$-equivariant birational morphism $X\to\mathbb{P}^3$ that blows up a smooth curve $C=S_2\cap S_3$, where $S_2$ is an irreducible quadric surface, and $S_3$ is an irreducible cubic surface. In particular, we may consider $A$ as a subgroup of $\mathrm{PGL}_4(\mathbb{C})$. Note that $C$ is contained in the smooth locus of the surface $S_2$, and $S_2$ is $A$-invariant, because $S_2$ is the unique quadric surface in $\mathbb{P}^3$ that contains $C$. Let $\widetilde{A}$ be a finite subgroup in $\mathrm{GL}_{4}(\mathbb{C})$ that is mapped to $A$ via the natural projection $\mathrm{GL}_{4}(\mathbb{C})\to\mathrm{PGL}_4(\mathbb{C})$. If $\widetilde{A}\simeq A$, then $A$ fixes a point in $\mathbb{P}^3$, so it also fixes a point in $X$ by \cite[Proposition~A.4]{ReYou00}. However, a priori, $\widetilde{A}$ is a central extension of the group $A$, which may not be abelian. In any case, we have the following exact sequence of $\widetilde{A}$-representations:
$$
0\longrightarrow H^0\big(\mathcal{O}_{\mathbb{P}^3}(3)\otimes\mathcal{I}_C\big)\longrightarrow H^0\big(\mathcal{O}_{\mathbb{P}^3}(3)\big) 
\longrightarrow H^0\big(\mathcal{O}_{S_2}(C)\big)\longrightarrow 0,
$$
in which $\mathcal{I}_C$ is the ideal sheaf of $C$. Hence, since $H^0\big(\mathcal{O}_{S_2}(C)\big)$ has a one-dimensional subrepresentation corresponding to $C$, we see that  
$H^0\big(\mathcal{O}_{\mathbb{P}^3}(3)\big)$ also has a corresponding one-dimensional subrepresentation.
Thus, we may choose $S_3$ to be $A$-invariant. Then, it follows from Duncan's lemma \cite[Lemma~2.4]{AbbanCheltsovKishimotoMangolte-2} that we can also choose $\widetilde{A}$ to be isomorphic to $A$, so, as explained earlier, the group $A$ fixes a point in $X$.  

Finally, we assume that $X$ is contained in the family \textnumero 5.1. 
This family contains one smooth member, which can be described as follows.
Let $Q$ be the~smooth quadric 3-fold
$$
\{x_1x_2+x_2x_3+x_3x_1+yz=0\}\subset\mathbb{P}^4,
$$
where $x_1$, $x_2$, $x_3$, $y$, $z$ are coordinates on $\mathbb{P}^4$.
Let $C$ be the~smooth conic in $Q$ that is cut out by  $y=z=0$, and let
$P_1=[1:0:0:0:0]$, $P_2=[0:1:0:0:0]$, $P_3=[0:0:1:0:0]$, all contained in $C$.
Let~$\theta\colon Y\to Q$ be the blowup of the points $P_1$, $P_2$,~$P_3$,
and $\widetilde{C}$ be the~strict transform on $Y$ of the~conic~$C$.
Then there is a birational morphism $\eta\colon X\to Y$ that blows up $\widetilde{C}$.
Note that the group $\mathrm{Aut}(X)$ is explicitly described in \cite[\S~5.23]{Book}, and this description implies that 
$$
\mathrm{Aut}(X)\simeq\mathfrak{S}_3\times\big(\mathbb{C}^\ast\rtimes\mathbb{Z}/2\mathbb{Z}\big),
$$ 
both birational morphism $\theta$ and $\eta$ are $\mathrm{Aut}(X)$-equivariant,
and we have the following $\mathrm{Aut}(X)$-equivariant commutative diagram:
$$
\xymatrix@R=1em{
&& X\ar@{->}[dl]_{\eta}\ar@{->}[dr]^{\eta^\prime}&&\\
&Y\ar@{->}[dl]_{\theta}\ar@{->}[dr]_{\phi}\ar@{-->}[rr]&& Y^\prime\ar@{->}[dl]^{\phi^\prime}\ar@{->}[dr]^{\nu}\\%
Q && V&& \mathbb{P}^1}
$$
where $V$ is the singular K-polystable Fano 3-fold constructed in \cite[\S~6]{one-dim},
$\phi$ is the small contraction of the curve $\widetilde{C}$ to the singular point of $V$,
$\phi^\prime$ is another small resolution of $V$, $\eta^\prime$ is the contraction of the $\eta$-exceptional divisor such that 
the dashed arrow is the Atiyah flop of the curve $\widetilde{C}$, and $\nu$ is a fibration whose general fiber is a sextic del Pezzo surface.
In particular, the group $A$ acts on the conic $C$ such that the subset $\{P_1,P_2,P_3\}$ is $A$-invariant.
Then $A$ fixes a point in $C$ by Duncan's lemma \cite[Lemma~2.4]{AbbanCheltsovKishimotoMangolte-2},
so $A$ fixes a point in $X$ by \cite[Proposition~A.4]{ReYou00}.
\end{proof}

Recall that the deformation families
\begin{center}
\textnumero 2.34, \textnumero 3.27, \textnumero 3.28, \textnumero 4.10, \textnumero 5.3, \textnumero 6.1, \textnumero 7.1, \textnumero 8.1, \textnumero 9.1, \textnumero 10.1
\end{center}
consist of products $\mathbb{P}^1\times S$,
where $S$ is a smooth del Pezzo surface, so their automorphism groups always contain a subgroup isomorphic to $(\mathbb{Z}/2\mathbb{Z})^2$ that acts trivially on the second factor,
and does not fix points. Moreover, it has been shown in \cite{AbbanCheltsovKishimotoMangolte-2} that any smooth member of the families
\textnumero 2.33, \textnumero 3.31, \textnumero 4.8, \textnumero 4.12, \textnumero 5.2
does not satisfy Condition~$\mathbf{(A)}$, and every deformation family among
\begin{center}
\textnumero 1.9, \textnumero 2.5, \textnumero 2.10, \textnumero 2.12, \textnumero 2.16, \textnumero 2.21, \textnumero 2.23, \textnumero 2.24, \\
\textnumero 3.2, \textnumero 3.5, \textnumero 3.6, \textnumero 3.7, \textnumero 3.10, \textnumero 3.12, \textnumero 3.13, \textnumero 4.13,
\end{center}
contains a smooth Fano 3-fold that does not satisfy Condition~$\mathbf{(A)}$.
Hence, to complete the proof of Theorem\,A, we may assume that $X$ is contained in one of the families
\begin{center}
\textnumero 1.1, \textnumero 1.2, \textnumero 1.3, \textnumero 1.4, \textnumero 1.5, \textnumero 1.6, \textnumero 1.7, \textnumero 1.8, \textnumero 1.12, \textnumero 1.13, \\
\textnumero 1.14, \textnumero 1.17, \textnumero 2.2, \textnumero 2.3, \textnumero 2.4, \textnumero 2.6, \textnumero 2.7, \textnumero 2.8, \textnumero 2.18, \textnumero 2.19, \\
\textnumero 2.25, \textnumero 2.27, \textnumero 2.29,  \textnumero 2.32, \textnumero 3.1, \textnumero 3.3, \textnumero 3.4, \textnumero 3.9, \textnumero 3.17,\\ \textnumero 3.19,  \textnumero 3.20, \textnumero 3.25, \textnumero 4.1, \textnumero 4.2,   \textnumero~4.3, \textnumero 4.4, \textnumero 4.6, \textnumero 4.7.
\end{center}
In the remaining part of the section, we will provide examples of $X$ in each of these families together with an abelian group $A$ that does not fix points in $X$.
If every smooth member of the family contains such abelian group, we indicate it for clarity. Note that in some cases, $X$ would be the only smooth member of the family. 

\begin{example}
\label{example:A-1-1-1-12}
Fix $d\in\{2,3\}$. Let $X$ be the hypersurface 
$$
\big\{y^2+x_1^{2d}+x_2^{2d}+x_3^{2d}+x_4^{2d}=0\big\}\subset\mathbb{P}(1_{x_1},1_{x_2},1_{x_3},1_{x_4},d_{y}).
$$
Then $X$ is a smooth Fano 3-fold. If $d=3$, then $X$ belongs to the family \textnumero 1.1. If $d=2$, then $X$ belongs to the family \textnumero 1.12. 
Let $A$ be the subgroup in $\mathrm{Aut}(X)$ generated by the transformations that change signs of the~coordinates $x_1,\ldots,x_4$.
Then $A\simeq(\mathbb{Z}/2\mathbb{Z})^4$ and $A$ does not fix points in $X$.
\end{example}

\begin{example}
\label{example:A-1-2}
Fix $d\in\{2,3,4\}$. Let $X$ be the hypersurface  
$$
\big\{x_1^d+x_2^d+x_3^d+x_4^d+x_5^d=0\big\}\subset\mathbb{P}^4_{x_1,x_2,x_3,x_4,x_5}.
$$
Then $X$ is a smooth Fano 3-fold. If $d=2$, then $X$ is the unique smooth Fano 3-fold in the family \textnumero 1.16.
If $d=3$, then $X$ is contained in the family \textnumero 1.13. 
If $d=4$, then $X$ is contained in the family \textnumero 1.2.
Let $A$ be the subgroup in $\mathrm{Aut}(X)$ generated by the transformations that multiply coordinates $x_1,\ldots,x_4$ by primitive $d$-th root of unity.
Then $A\simeq(\mathbb{Z}/d\mathbb{Z})^4$ and $A$ does not fix points in $X$.
\end{example}

\begin{example}
\label{example:A-1-3}
Let $X$ be the~complete intersection
$$
\Big\{\sum_{i=0}^6x_i=\sum_{i=0}^6x_i^2=\sum_{i=0}^6x_i^3=0\Big\} \subset \mathbb P^6.
$$
Then $X$ is a~smooth Fano 3-fold in the deformation family \textnumero 1.3, and $\mathrm{Aut}(X)\simeq\mathfrak{S}_7$.
Let $A$ be the subgroup in $\mathrm{Aut}(X)$ generated by the following transformations:
\begin{align*}
[x_0:x_1:x_2:x_3:x_4:x_5:x_6]&\mapsto[x_1:x_2:x_0:x_3:x_4:x_5:x_6],\\
[x_0:x_1:x_2:x_3:x_4:x_5:x_6]&\mapsto[x_0:x_1:x_2:x_4:x_5:x_6:x_3].
\end{align*}
Then $A\simeq(\mathbb{Z}/3\mathbb{Z})\times (\mathbb{Z}/4\mathbb{Z})$ and $A$ does not fix points in $X$.
\end{example}

\begin{example}
\label{example:A-1-4}
Let $X$ be the~complete intersection
$$
\Big \{\sum_{i=1}^7x_i^2=\sum_{i=1}^7ix_i^2=\sum_{i=1}^72^ix_i^2=0\Big\} \subset \mathbb{P}^6,
$$
where $x_1,\ldots,x_7$ are coordinates on $\mathbb{P}^6$.
Then $X$ is a smooth Fano 3-fold in the family \textnumero 1.4.
Let $A$ be the subgroup in $\mathrm{Aut}(X)$ generated by the transformations that change signs of the~coordinates $x_1,\ldots,x_6$.
Then $A\simeq(\mathbb{Z}/2\mathbb{Z})^6$ and $A$ does not fix points in $X$.
\end{example}

\begin{example}[Gushel-Mukai 3-folds]
\label{example:A-1-5}
Let $V_5$ be a~smooth intersection of the~Grassmannian $\mathrm{Gr}(2,5)\subset\mathbb{P}^9$ in its Pl\"ucker embedding
with a~linear subspace of codimension $3$. Then $V_5$ is the~unique smooth Fano 3-fold in the deformation family \textnumero 1.15.
Moreover, it is well known that 
$$
\mathrm{Aut}(V_5)\cong\mathrm{PGL}_2(\mathbb{C}).
$$
Let $G$ be a subgroup in $\mathrm{Aut}(V_5)$ such that $G\simeq\mathfrak{A}_5$.
Then it follows from \cite[Theorem~8.2.1]{CheltsovShramov} that $|-K_{V_5}|$ contains a~pencil $\mathcal{P}$ such that every surface of $\mathcal{P}$ is $G$-invariant, $G$ acts faithfully on every surface in $\mathcal{P}$, and general surface in $\mathcal{P}$ is a smooth K3 surface.
Let $S$ be a smooth surface in $\mathcal{P}$,
let $\pi\colon X\to V_5$ be the~double cover branched over $S$,
and let $\tau\in\mathrm{Aut}(X)$ be the Galois involution of this double cover.
Then $X$ is a~smooth Fano 3-fold in the~family \textnumero 1.5.
Since the action of the group $G$ lifts to $X$, we identify $G$ with a subgroup in $\mathrm{Aut}(X)$. 
Observe that 
$$
\langle\tau,G\rangle\simeq (\mathbb{Z}/2\mathbb{Z})\times\mathfrak{A}_5.
$$
Now, let $A^\prime$ be a subgroup in $G$ such that $A^\prime\simeq(\mathbb{Z}/2\mathbb{Z})^2$,
and let $A=\langle\tau,A^\prime\rangle\simeq (\mathbb{Z}/2\mathbb{Z})^3$.
Then $A$ does not fix points in $X$.
Indeed, if $p$ is an $A$-fixed point in $X$, then $\pi(p)$ is an $A^\prime$-fixed point in $S$,
so the length of the $G$-orbit of $\pi(p)$ is $1$, $5$ or $15$,
which contradicts \cite[Lemma~6.7.1]{CheltsovShramov}.
\end{example}

\begin{example}[{\cite{Beauville}}]
\label{example:A-1-6}
Let $X$ be the~smooth Fano 3-fold constructed in \cite[Example~2.11]{Prokhorov-Simple}.
Then $X$ belongs to the~deformation family \textnumero 1.6 and $\mathrm{Aut}(X)\cong\mathrm{SL}_2(\mathbf{F}_8)$,
which is a~simple group. Let $A$ be a Sylow $2$-subgroup in $\mathrm{Aut}(X)$.
Then $A\simeq (\mathbb{Z}/2\mathbb{Z})^3$, and it follows from \cite{Beauville} that $A$ does not fix points in $X$.
\end{example}

\begin{example}
\label{example:A-1-7}
Let $X$ be the smooth Fano 3-fold in the family \textnumero 1.7 such that $\mathrm{Aut}(X)$ contains a subgroup $G\simeq (\mathbb{Z}/3\mathbb{Z})\rtimes\mathfrak{D}_8$ with GAP ID [24,8]
which has been constructed in \cite[\S~5]{TschinkelZhang}.
Then, up to conjugation, the group $G$ contains two abelian subgroups isomorphic to $(\mathbb{Z}/2\mathbb{Z})^2$. One of them is normal, and another one is not normal. Moreover, it follows from \cite[\S~5]{TschinkelZhang} that the non-normal subgroup does not fix points in $X$, which gives another proof of \cite[Proposition~5.1]{TschinkelZhang}.
\end{example}

\begin{example}
\label{example:A-1-8}
Let $V$ be the Lagrangian Grassmannian $\mathrm{LGr}(3,6)$ embedded by Plucker into $\mathbb{P}^{13}$.
Set
$$
X=\left(
  \begin{array}{ccc}
    x_{11} & x_{12} & x_{13} \\
    x_{12} & x_{22} &  x_{23}\\
    x_{13} & x_{23} & x_{33} \\
  \end{array}
\right), Y=\left(
  \begin{array}{ccc}
    y_{11} & y_{12} & y_{13} \\
    y_{12} & y_{22} & y_{23}\\
    y_{13} & y_{23} & y_{33} \\
  \end{array}
\right).
$$
Then it follows from \cite{IlievRanestad} that $V$ is given by $21$ quadratic equations, which can be described as follows:
$$
\left\{\aligned
&\mathrm{adj}(X)=uY, \\
&\mathrm{adj}(Y)=vX, \\
&XY=uvI_3,
\endaligned
\right.
$$
where $\mathrm{adj}(X)$ and $\mathrm{adj}(Y)$ are adjoint matrices of $X$ and $Y$, respectively,
$I_3$ is the $3\times 3$ identity matrix, 
and $u$, $v$, $x_{11}$, $x_{22}$, $x_{33}$, $x_{12}$, $x_{13}$, $x_{23}$, $y_{11}$, $y_{22}$, $y_{33}$, $y_{12}$, $y_{13}$, $y_{23}$ are coordinates on $\mathbb{P}^{13}$.
Let $A$ be the subgroup in $\mathrm{Aut}(V)$ generated by the following involutions:  
\begin{align*}
[u:v:x_{11}:x_{22}:x_{33}:x_{12}:x_{13}:x_{23}:y_{11}:y_{22}:y_{33}:y_{12}:y_{13}:y_{23}]&\mapsto\\ 
[u:v:x_{11}:x_{22}:x_{33}:-x_{12}:x_{13}:-x_{23}:y_{11}:y_{22}:y_{33}:-y_{12}:y_{13}:-y_{23}],&\\
[u:v:x_{11}:x_{22}:x_{33}:x_{12}:x_{13}:x_{23}:y_{11}:y_{22}:y_{33}:y_{12}:y_{13}:y_{23}]&\mapsto\\ 
[u:v:x_{11}:x_{22}:x_{33}:-x_{12}:-x_{13}:x_{23}:y_{11}:y_{22}:y_{33}:-y_{12}:-y_{13}:y_{23}],&\\
[u:v:x_{11}:x_{22}:x_{33}:x_{12}:x_{13}:x_{23}:y_{11}:y_{22}:y_{33}:y_{12}:y_{13}:y_{23}]&\mapsto\\ 
[v:u:y_{11}:y_{22}:y_{33}:y_{12}:y_{13}:y_{23}:x_{11}:x_{22}:x_{33}:x_{12}:x_{13}:x_{23}].&
\end{align*}
Then $A\simeq(\mathbb{Z}/2\mathbb{Z})^3$. Let $X$ be the 3-fold in $V$ that is cut out by 
$$
\left\{\aligned
&1967x_{11}+1973x_{22}+1983x_{33}=0,\\ 
&1967y_{11}+1973y_{22}+1983y_{33}=0,\\
&2024x_{11}+2025x_{22}+2024y_{11}+2025y_{22}=v+u.
\endaligned
\right.
$$
Then $X$ is $A$-invariant smooth Fano 3-fold in the family \textnumero 1.8,
and $A$ fix no points in $X$.
\end{example}

\begin{example}
\label{example:A-1-14}
Let $X$ be a smooth Fano 3-fold in family \textnumero 1.14. Then $X$ is a complete intersection of two quadrics in $\mathbb{P}^5$,
and it follows from \cite{Reid1972} that we can choose coordinates $x_1,x_2,x_3,x_4,x_5,x_6$ on $\mathbb{P}^5$ such that $X$ is given by
$$
\sum_{i=1}^6x_i^2=\sum_{i=1}^6a_ix_i^2=0
$$
for some numbers $a_1,a_2,a_3,a_4,a_5,a_6$.
Let $A$ be the subgroup in $\mathrm{Aut}(X)$ generated by the transformations that change signs of the~coordinates $x_1,\ldots,x_5$.
Then $A\simeq(\mathbb{Z}/2\mathbb{Z})^5$ and $A$ does not fix points in $X$.
\end{example}

\begin{example}
\label{example:A-1-17-2-27-3-25-4-6}
Recall that $\mathbb{P}^3$ is the only smooth Fano 3-fold in the family \textnumero 1.17.
Let $A$ be the subgroup in $\mathrm{Aut}(\mathbb{P}^3_{x_1,x_2,x_3,x_4})$ generated by the following transformations:
\begin{align*}
[x_1:x_2:x_3:x_4]&\mapsto[-x_1:x_2:-x_3:x_4],\\
[x_1:x_2:x_3:x_4]&\mapsto[x_2:x_1:x_4:x_3].
\end{align*}
Then $A\simeq(\mathbb{Z}/2\mathbb{Z})^2$ and $A$ does not fix points in $\mathbb{P}^3$.
Let $Z$ be an $A$-invariant smooth subvariety in $\mathbb{P}^3$, and let $X$ be the blow up of $\mathbb{P}^3$ along $Z$. 
Then the action of $A$ lifts to $X$, and $A$ does not fix points in~$X$. 
Now, choosing appropriate $Z$, we see that all smooth members of the families \textnumero 2.25, \textnumero 2.27, \textnumero 3.25, \textnumero 4.6
also do not satisfy Condition~$\mathbf{(A)}$. 
Namely, if $Z$ is the smooth quartic elliptic curve 
$$
\big\{x_1^2+x_2^2+\lambda(x_3^2+x_4^2)=0, \lambda(x_1^2-x_2^2)+x_3^2-x_4^2=0\big\}\subset\mathbb{P}^3
$$
for $\lambda\not\in\{0,\pm 1,\pm i\}$, then $X$ is a smooth Fano 3-fold in the family \textnumero 2.25,
and every smooth member of this family can be obtained in this way.
Likewise, if $Z$ is the twisted cubic $\varphi(\mathbb{P}^1)$ for $\varphi\colon \mathbb{P}^1\to \mathbb{P}^3$ given by $$
[u:v]\mapsto [uv^2:u^2v:u^3:v^3],
$$
then $X$ is the unique smooth member of the family \textnumero 2.27.
Similarly, if $Z=\{x_1=x_2=0\}\cup\{x_3=x_4=0\}$, then $X$ is the unique smooth member of the family \textnumero 3.25.
Finally, if $Z$ is the union of the three disjoint lines $\{x_1=x_2=0\}$, $\{x_3=x_4=0\}$, $\{x_1+x_3=x_2+x_4=0\}$, then $X$ is the member of the family \textnumero 4.6.
\end{example}

\begin{example}
\label{example:A-2-2-2-18}
Fix $d\in\{2,4\}$. Let
$$
S=\Big\{u^2\big(1967x^{d}+1973y^{d}+1983z^{d}\big)+v^2\big(1983x^{d}+1973y^{d}+1967z^{d}\big)=0\Big\}\subset \mathbb{P}^1_{u,v}\times\mathbb{P}^2_{x,y,z}.
$$
Then $S$ is a smooth surface. Let $\pi\colon X\to\mathbb{P}^1_{u,v}\times\mathbb{P}^2_{x,y,z}$ be a~double cover branched over $S$.
If $d=2$, then $X$ is a~smooth Fano 3-fold in the~deformation family \textnumero 2.18.
If $d=4$, then $X$ is a~smooth Fano 3-fold in the~deformation family \textnumero 2.2.
Let $A^\prime$ be the subgroup in $\mathrm{Aut}(\mathbb{P}^1_{u,v}\times\mathbb{P}^2_{x,y,z})$ generated by 
\begin{align*}
\big([u:v],[x:y:z]\big)&\mapsto\big([u:-v],[x:y:z]\big),\\
\big([u:v],[x:y:z]\big)&\mapsto\big([u:v],[\omega x:y:z]\big),\\
\big([u:v],[x:y:z]\big)&\mapsto\big([u:v],[x:\omega y:z]\big),
\end{align*}
where $\omega$ is a primitive $d$-th root of unity.
Then $A^\prime\simeq (\mathbb{Z}/2\mathbb{Z})\times(\mathbb{Z}/d\mathbb{Z})^2$, the action of the group $A^\prime$ lifts to $X$, so we can identify $A^\prime$ with a subgroup in $\mathrm{Aut}(X)$.  Let $\tau$ be the~Galois involution of the~double cover~$\pi$, and let $A=\langle\tau,A^\prime\rangle$. Then $A\simeq(\mathbb{Z}/2\mathbb{Z})^2\times(\mathbb{Z}/d\mathbb{Z})^2$, and $A$ does not fix points in $X$.
\end{example}

\begin{example}
\label{example:A-2-3}
Let us use the notation and assumptions of Example~\ref{example:A-1-1-1-12} with $d=2$. 
Let $C$ be the smooth elliptic curve in $X$ that is cut out by $x_3=x_4=0$, 
and let $Y$ be the blow up of $X$ along $C$. Then $Y$ is a~smooth Fano 3-fold in the~family \textnumero 2.3, 
and the curve $C$ is $A$-invariant, so the action of the group $A$ lifts to $Y$. Note that $A$ does not fix points in $Y$.
\end{example}

\begin{example}
\label{example:A-2-4}
Let $C$ be the~curve in $\mathbb{P}^3$ that is given by
$$
\left\{\aligned
&x_1^3+x_2^3+\lambda(x_3^3+x_4^3)=0,\\
&\lambda(x_1^3-x_2^3)+x_3^3-x_4^3=0,\\
\endaligned
\right.
$$
where $\lambda$ is a general complex number. Then $C$ is a~smooth curve.
Let $X\to\mathbb{P}^3$ be a~blow up of this curve.
Then $X$ is a~smooth Fano 3-fold in the family \textnumero 2.4.
Let $A$ be the subgroup in $\mathrm{Aut}(\mathbb{P}^3)$ generated by 
\begin{align*}
[x_1:x_2:x_3:x_4]&\mapsto[x_2:x_1:x_4:x_3],\\
[x_1:x_2:x_3:x_4]&\mapsto[x_3:-x_4:x_1:-x_2].
\end{align*}
Then $A\simeq(\mathbb{Z}/2\mathbb{Z})^2$, and $A$ does not fix points in $\mathbb{P}^3$.
Note that the curve $C$ is $A$-invariant, so its action lifts to $X$, and $A$ does not fix points in $X$.
\end{example}

\begin{example}[Verra 3-folds]
\label{example:A-2-6}
Let $X$ be the divisor of degree $(2,2)$ in $\mathbb{P}^2_{u,v,w}\times\mathbb{P}^2_{x,y,z}$ that is given by 
$$
vwx^2+uwy^2+uvz^2+yzu^2+xzv^2+xyw^2+\lambda(u^2x^2+v^2y^2+w^2z^2)=0,
$$
where $\lambda$ is a sufficiently general complex number. Then $X$ is a smooth Fano 3-fold in the family \textnumero 2.6.
Let $A$ be the subgroup in $\mathrm{Aut}(X)$ that is generated by the following transformations:
\begin{align*}
\big([u:v:w],[x:y:z]\big)&\mapsto\big([v:w:u],[y:z:x]\big),\\
\big([u:v:w],[x:y:z]\big)&\mapsto\big([u:\omega v:\omega^2w],[\omega^2x:\omega y:z]\big),
\end{align*}
where $\omega$ is a primitive cube root of unity. 
Then $A\simeq(\mathbb{Z}/3\mathbb{Z})^2$, and $A$ does not fix points in $X$.
\end{example}

\begin{example}
\label{example:A-2-7}
Let $Q$ be the smooth quadric 3-fold $\{x_1^{2}+x_2^{2}+x_3^{2}+x_4^{2}+x_5^{2}=0\}\subset\mathbb{P}^4_{x_1,x_2,x_3,x_4,x_5}$,
and let $A$ be the subgroup in $\mathrm{Aut}(Q)$ generated by the following involutions:
\begin{align*}
[x_1:x_2:x_3:x_4:x_5]&\mapsto[-x_1:x_2:x_3:x_4:x_5]\\
[x_1:x_2:x_3:x_4:x_5]&\mapsto[x_1:-x_2:-x_3:x_4:x_5]\\
[x_1:x_2:x_3:x_4:x_5]&\mapsto[x_1:-x_2:x_3:-x_4:x_5].
\end{align*}
Then $A\simeq(\mathbb{Z}/2\mathbb{Z})^3$.
Let $Z$ be a smooth $A$-invariant subvariety of the quadric $Q$, and let $X$ be the blow up of $Q$ along $Z$. Then the action of $A$ lifts to $X$, and $A$ does not fix points in $X$.
Now, choosing appropriate $Z$, we obtain smooth Fano 3-folds  of the families \textnumero 2.7, \textnumero 2.29, \textnumero 3.19, \textnumero 3.20
that do not satisfy Condition~$\mathbf{(A)}$. Namely, if $Z$ is the smooth curve of genus $5$ that is cut out by
$$
\sum_{i=1}^{5}ix_i^2=\sum_{i=1}^{5}2^ix_i^2=0,
$$
then $X$ is a smooth Fano 3-fold in the family \textnumero 2.7.
Similarly, if $Z$ is the conic $Q\cap\{x_3=x_4=0\}$, then $X$ is the~unique smooth member of the~family \textnumero 2.29.
If
$$
Z=\{x_1=x_2+ix_3=x_4+ix_5=0\}\cup \{x_1=x_2-ix_3=x_4-ix_5=0\},
$$
then $X$ is the unique smooth Fano 3-fold in the family \textnumero 3.20.
Finally, we let
$$
Z=[0:0:0:1:i]\cup [0:0:0:1:-i].
$$
Then $X$ is the unique member of the family \textnumero 3.19.
Let $C$ be the $A$-invariant conic $Q\cap\{x_3=x_4=0\}$,
and let $Y$ be the blow up of $X$ along the strict transform of $C$. Then $Y$ is the unique smooth Fano 3-fold in the family \textnumero 4.4, the action of $A$ lifts to $Y$, and $A$ does not fix points in $Y$.
\end{example}

\begin{example}
\label{example:A-2-8}
Let $Y$ be the hypersurface 
$$
\{w^2+t^2\big(x^2+y^2+z^2\big)+x^4+y^4+z^4=0\}\subset\mathbb{P}(1_{x},1_{y},1_{z},1_{t},2_{w}).
$$
Then $\mathrm{Sing}(Y)=[0:0:0:1:0]$, and $Y$ has an ordinary double singularity at the point $[0:0:0:1:0]$.
Let $\pi\colon X\to Y$ be the blow up of this point. Then $X$ is a smooth Fano 3-fold in the~family \textnumero 2.8.
Let $A$ be the subgroup in $\mathrm{Aut}(Y)$ generated by the transformations that change signs of the~coordinates $x,y,z,t$.
Then $A\simeq(\mathbb{Z}/2\mathbb{Z})^4$, the action of the group $A$ lifts to $X$, and $A$ does not fix points in $X$.
\end{example}

\begin{example}
\label{example:A-2-19}
Let $S_2$ be the surface $\{x_0x_3=x_1x_2\}\subset\mathbb{P}^3$.
Fix the isomorphism $\mathbb{P}^1\times\mathbb{P}^1\simeq S_2$ by $([s_0:s_1],[t_0:t_1])\mapsto[s_0t_0:s_0t_1:s_1t_0:s_1t_1]$, and let $C$ be the curve in $S_2$ given by
$$
(s_0^2+s_1^2)(t_0^3+t_1^3)+\varepsilon (s_0^2-s_1^2)(t_0^3-t_1^3)=0
$$
in which $\varepsilon$ is a general number so that $C$ is smooth.
Let $A$ be the subgroup in $\mathrm{Aut}(S_2)$ generated by
\begin{align*}
\big([s_0:s_1],[t_0:t_1]\big)&\mapsto\big([s_0:-s_1],[t_0:t_1]\big),\\
\big([s_0:s_1],[t_0:t_1]\big)&\mapsto\big([s_1:s_0],[t_1:t_0]\big).
\end{align*}
Then $A\simeq(\mathbb{Z}/2\mathbb{Z})^2$, the~curve $C$ is $A$-invariant,
and the $A$-action extends to $\mathbb{P}^3$ as follows:
\begin{align*}
[x_0:x_1:x_2:x_3]&\mapsto[-x_0:-x_1:x_2:x_3],\\
[x_0:x_1:x_2:x_3]&\mapsto[x_3:x_2:x_1:x_0].
\end{align*}
Hence, $A$ fixes no points in $\mathbb{P}^3$.
Let $\pi\colon X\to\mathbb{P}^3$ be the blowup of the curve $C$. 
Then $X$ is a smooth Fano 3-fold in the family \textnumero 2.19, 
the $A$-action lifts to $X$, and $A$ does not fix points in $X$.
\end{example}

\begin{example}
\label{example:A-2-32-4-7}
Let $W=\{x_1y_1+x_2y_2+x_3y_3=0\}\subset\mathbb{P}^2_{x_1,x_2,x_3}\times\mathbb{P}^2_{y_1,y_2,y_3}$.
Then $W$ is the unique smooth Fano 3-fold in the family \textnumero 2.32.
Now, let $A$ be the~subgroup in $\mathrm{Aut}(W)$ generated by 
\begin{align*}
([x_1:x_2:x_3],[y_1:y_2:y_3])&\mapsto([x_2:x_3:x_1],[y_2:y_3:y_1]),\\
([x_1:x_2:x_3],[y_1:y_2:y_3])&\mapsto([\omega^2x_1:\omega x_2:x_3],[\omega y_1:\omega^2y_2:y_3]),
\end{align*}
where $\omega$ is a primitive cube root of unity. Then $A\simeq(\mathbb{Z}/3\mathbb{Z})^2$, and $A$ fixes no points in $W$.
Alternatively, consider the subgroup $A^\prime\subset\mathrm{Aut}(W)$ such that $A^\prime\simeq(\mathbb{Z}/2\mathbb{Z})^2$ and $A^\prime$ is generated by \begin{align*}
([x_1:x_2:x_3],[y_1:y_2:y_3])&\mapsto([-x_1:x_2:x_3],[-y_1:y_2:y_3]),\\
([x_1:x_2:x_3],[y_1:y_2:y_3])&\mapsto([x_1:-x_2:x_3],[y_1:-y_2:y_3]).
\end{align*} 
Then $A^\prime$ also fixes no points in $W$.
Let $C_1=\{x_1=y_2=y_3=0\}$ and  $C_2=\{y_1=x_2=x_3=0\}$. Then the curve $C_1+C_2$ is smooth and $A^\prime$-invariant.
Let $X\to W$ be the blow up of this curve. Then $X$ is the unique smooth Fano 3-fold in the family \textnumero 4.7,
the action of the group $A^\prime$ lifts to $X$, and $A^\prime$ does not fix points in $X$.
\end{example}

\begin{example}
\label{example:A-3-1}
Let $S$ be the surface of degree $(2,2,2)$ in $\mathbb{P}^1_{x_1,y_1}\times\mathbb{P}^1_{x_2,y_2}\times\mathbb{P}^1_{x_3,y_3}$ that is given by  
\begin{multline*}
\quad\quad\quad \quad \quad x_1^2x_2y_2x_3^2+y_1^2x_2y_2y_3^2+x_1^2x_2^2x_3y_3+x_1y_1x_2^2x_3^2++y_1^2y_2^2x_3y_3+x_1y_1y_2^2y_3^2=\\
=2025(y_1^2y_2^2x_3y_3+x_1y_1y_2^2y_3^2+y_1^2x_2y_2y_3^2+x_1^2x_2^2x_3y_3+ x_1^2x_2x_3^2y_2+x_1x_2^2x_3^2y_1).\quad\quad\quad\quad\quad
\end{multline*}
Then $S$ is smooth. Let $\pi\colon X\to \mathbb{P}^1_{x_1,y_1}\times\mathbb{P}^1_{x_2,y_2}\times\mathbb{P}^1_{x_3,y_3}$ be the double cover that is ramified in $S$.
Then $X$ is a smooth Fano 3-fold in the family \textnumero 3.1.
Let $A^\prime$ be the subgroup in $\mathrm{Aut}(\mathbb{P}^1_{x_1,y_1}\times\mathbb{P}^1_{x_2,y_2}\times\mathbb{P}^1_{x_3,y_3})$ generated by 
\begin{align*}
([x_1:y_1],[x_2:y_2],[x_3:y_3])&\mapsto([y_1:x_1],[y_2:x_2],[y_3:x_3]),\\
([x_1:y_1],[x_2:y_2],[x_3:y_3])&\mapsto([x_2:y_2],[x_3:y_3],[x_1:y_1]).
\end{align*}
Then $A^\prime\simeq(\mathbb{Z}/2\mathbb{Z})\times (\mathbb{Z}/3\mathbb{Z})$, the surface $S$ is $A^\prime$-invariant, and $A^\prime$ does not fix points in $S$.
Observe that the action of the group $A^\prime$ lifts to $X$, so we can consider $A^\prime$ as a subgroup in $\mathrm{Aut}(X)$.
Let $\tau$ be the Galois involution of $\pi$, and let $A=\langle\tau,A\rangle$.
Then $A\simeq(\mathbb{Z}/2\mathbb{Z})^2\times (\mathbb{Z}/3\mathbb{Z})$, and $A$ does not fix points in $X$.
\end{example}

\begin{example}
\label{example:A-3-3}
Let $X$ be the~3-fold in $\mathbb{P}^3_{x_1,y_1,z_1,w_1}\times\mathbb{P}^3_{x_2,y_2,z_2,w_2}$ given by
$$
\left\{\aligned
&x_1x_2^2+y_1y_2^2+z_1z_2^2+w_1w_2^2=0, \\
&x_1^2+y_1^2+z_1^2+w_1^2=0, \\
&x_2+y_2+z_2+w_2=0.
\endaligned
\right.
$$
Then $X$ is smooth Fano 3-fold \textnumero 3.3.
Let $A$ be the subgroup in $\mathrm{Aut}(X)$ generated by
\begin{align*}
([x_1:y_1:z_1:w_1],[x_2:y_2:z_2:w_2])&\mapsto([y_1:x_1:w_1:z_1],[y_2:x_2:w_2:z_2]),\\
([x_1:y_1:z_1:w_1],[x_2:y_2:z_2:w_2])&\mapsto([w_1:z_1:y_1:x_1],[w_2:z_2:y_2:x_2]).
\end{align*}
Then $A\simeq(\mathbb{Z}/2\mathbb{Z})^2$, and $A$ does not fix points in $X$. 
\end{example}

\begin{example}
\label{example:A-3-4}
Let us use the notation and assumptions of Example~\ref{example:A-2-2-2-18} with $d=2$.
Let $C$ be the preimage via $\pi$ of the curve $\{y=0,z=0\}\subset \mathbb{P}^1_{u,v}\times\mathbb{P}^2_{x,y,z}$.
Then $C$ is smooth and $A$-invariant. Let $Y$ be the blow up of the 3-fold $X$ along the curve $C$.
Then $Y$ is a smooth Fano 3-fold in the family \textnumero 3.4, and the action of the group $A$ lifts to $Y$.
Since $A$ does not fix points in $X$, we see that $A$ does not fix points in $Y$.
\end{example}

\begin{example}
\label{example:A-3-9}
Let $S$, $E$, $E^\prime$ be surfaces in $\mathbb{P}^1_{u,v}\times \mathbb{P}^2_{x,y,z}$ defined as follows: 
$$
S=\{x^4+y^4+z^4+2025(x^2y^2+x^2z^2+y^2z^2)=0\},\quad E=\{u-iv=0\}, \quad E^\prime=\{u+iv=0\}.
$$
Then $S$, $E$, $E^\prime$ are smooth.
Let $A^\prime$ be the subgroup in $\mathrm{Aut}(\mathbb{P}^1_{u,v}\times \mathbb{P}^2_{x,y,z})$ generated by 
\begin{align*}
\big([u:v],[x:y:z]\big)&\mapsto \big([u:v],[-x:y:z]\big),\\
\big([u:v],[x:y:z]\big)&\mapsto \big([u:v],[x:-y:z]\big),\\
\big([u:v],[x:y:z]\big)&\mapsto \big([v:u],[x:y:z]\big).
\end{align*}
Then $A^\prime\simeq(\mathbb{Z}/2\mathbb{Z})^3$, and $S+E+E^\prime$ is $A^\prime$-invariant.
Let $\eta\colon\overline{X}\to \mathbb{P}^1_{u,v} \times \mathbb{P}^2_{x,y,z}$ be a~double cover branched over $S+E+E^\prime$,
and let $\overline{S}$, $\overline{E}$, $\overline{E}^\prime$ be the~preimages on $\overline{X}$ of the~surfaces $S$, $E$, $E^\prime$, respectively.
Then~the action of the group $A^\prime$ lifts to $\overline{X}$, so we consider $A^\prime$ as a subgroup in $\mathrm{Aut}(\overline{X})$.
Let $\tau$ be the~Galois involution of the~double cover $\eta$,
and let $A=\langle A^\prime,\tau\rangle$. Then $A\simeq(\mathbb{Z}/2\mathbb{Z})^4$,
and $A$ does not fix points in $\overline{X}$.
Note that $\overline{X}$ is singular along the curves $\overline{E}\cap\overline{S}$ and $\overline{E}^\prime\cap\overline{S}$.
But we can $A$-equivariantly blow up $\overline{X}$ along these curves to get a smooth 3-fold $\widehat{X}$.
Then there exists an~$A$-equivariant birational morphism $\widehat{X}\to X$ that contracts the strict transform of $\overline{S}$ to a~smooth curve of genus $3$, 
and $X$ is a smooth Fano 3-fold in the family \textnumero 3.9.
By construction, the group $A$ does not fix points in $X$.
\end{example}

\begin{example}
\label{example:A-3-17}
Let $X$ be the~unique smooth Fano 3-fold in the deformation family \textnumero 3.17.
Then 
$$
X=\big\{x_0y_0z_2+x_1y_1z_0=x_0y_1z_1+x_1y_0z_1\big\}\subset\mathbb{P}^1_{x_0,x_1}\times\mathbb{P}^1_{y_0,y_1}\times\mathbb{P}^2_{z_0,z_1,z_2},
$$
Let $A$ be the subgroup in $\mathrm{Aut}(X)$ generated by the following transformation:
\begin{align*}
\big([x_0:x_1],[y_0:y_1],[z_0:z_1:z_2]\big)&\mapsto\big([y_0:y_1],[x_0:x_1],[z_0:z_1:z_2]\big),\\
\big([x_0:x_1],[y_0:y_1],[z_0:z_1:z_2]\big)&\mapsto\big([x_1:x_0],[y_1:y_0],[z_2:z_1:z_0]\big),\\
\big([x_0:x_1],[y_0:y_1],[z_0:z_1:z_2]\big)&\mapsto\big([x_0:-x_1],[y_0:-y_1],[z_0:-z_1:z_2]\big).
\end{align*}
Then $A\simeq(\mathbb{Z}/2\mathbb{Z})^3$ and $A$ does not fix points in $X$.
\end{example}

\begin{example}
\label{example:A-4-1}
Let $X$ be a smooth Fano 3-fold in the family \textnumero 4.1. It follows from \cite{CheltsovFedorchukFujitaKaloghiros} that $X$ can be given in $\mathbb{P}^1_{x_1,y_1}\times\mathbb{P}^1_{x_2,y_2}\times\mathbb{P}^1_{x_3,y_3}\times\mathbb{P}^1_{x_4,y_4}$ by the~equation
\begin{multline*}
\quad\quad\quad \quad \quad \quad \quad \quad a\big(x_1x_2x_3x_4+y_1y_2y_3y_4\big)
+b\big(x_1x_2y_3y_4+y_1y_2x_3x_4\big)+\\
+c\big(x_1y_2x_3y_4+y_1x_2y_3x_4\big)
+d\big(x_1y_2y_3x_4+y_1x_2x_3y_4\big)=0\quad \quad \quad \quad \quad \quad \quad\quad
\end{multline*}
for some  numbers $a,b,c,d$. Let $A$ be the subgroup in $\mathrm{Aut}(X)$ generated by the~following transformations:
\begin{align*}
\big([x_1:y_1],[x_2:y_2],[x_3:y_3],[x_4:y_4]\big)&\mapsto\big([x_2:y_2],[x_1:y_1],[x_4:y_4],[x_3:y_3]\big),\\
\big([x_1:y_1],[x_2:y_2],[x_3:y_3],[x_4:y_4]\big)&\mapsto\big([x_4:y_4],[x_3:y_3],[x_2:y_2],[x_1:y_1]\big),\\
\big([x_1:y_1[,[x_2:y_2],[x_3:y_3],[x_4:y_4]\big)&\mapsto\big([y_1:x_1],[y_2:x_2],[y_3:x_3],[y_4:x_4]\big),\\
\big([x_1:y_1],[x_2:y_2],[x_3:y_3],[x_4:y_4]\big)&\mapsto\big([x_1:-y_1],[x_2:-y_2],[x_3:-y_3],[x_4:-y_4]\big).
\end{align*}
Then $A\simeq(\mathbb{Z}/2\mathbb{Z})^4$ and $A$ does not fix points in $X$.
\end{example}

\begin{example}
\label{example:A-4-2}
Let $X$ be any smooth Fano 3-fold in the family \textnumero 4.2.
Then there exists a birational morphism $\pi\colon X\to Q$ such that 
$Q$ is the cone $\{x_1^2+x_2^2+x_3^2+x_4^2=0\}\subset\mathbb{P}^4_{x_1,x_2,x_3,x_4,x_5}$,
and $\pi$ is a blow up of the point $[0:0:0:0:1]=\mathrm{Sing}(Q)$ and the smooth elliptic curve
$$
C=\{x_5=a_1x_1^2+a_2x_2^2+a_3x_3^2+a_4x_4^2=x_1^2+x_2^2+x_3^2+x_4^2=0\}\subset Q\setminus [0:0:0:0:1],
$$ 
where $a_1,a_2,a_3,a_4$ are some numbers. Let $A$ be the subgroup in $\mathrm{Aut}(Q)$ generated by the transformations that change signs of the~coordinates $x_1,\ldots,x_4$. Then $A\simeq(\mathbb{Z}/2\mathbb{Z})^4$, $[0:0:0:0:1]$ is the only $A$-fixed  point in $Q$,
and $C$ is $A$-invariant, so the $A$-action lifts to $X$. 
Moreover, $A$ does not fix points in $X$.
\end{example}

\begin{example}
\label{example:A-4-3}
Let $C$ be the~curve of degree $(1,1,2)$ in $\mathbb{P}^1_{x_0,x_1}\times\mathbb{P}^1_{y_0,y_1}\times\mathbb{P}^1_{z_0,z_1}$ given by
$$
\left\{\aligned
&x_0y_1-x_1y_0=0,\\
&x_0^2z_1+x_1^2z_0=0.
\endaligned
\right.
$$
Then $C$ is smooth and irreducible.
Let~$\pi\colon X\to \mathbb{P}^1\times\mathbb{P}^1\times\mathbb{P}^1$ be the~blow up of the curve $C$.
Then $X$ is the~unique smooth Fano 3-fold~\textnumero~4.3.
Let $A$ be the~subgroup of $\mathrm{Aut}(X)$ generated by 
\begin{align*}
\big([x_0:x_1],[y_0:y_1],[z_0:z_1]\big)&\mapsto\big([x_1:x_0],[y_1:y_0],[z_1:z_0]\big),\\
\big([x_0:x_1],[y_0:y_1],[z_0:z_1]\big)&\mapsto\big([y_0:y_1],[x_0:x_1],[z_0:z_1]\big),\\
\big([x_0:x_1],[y_0:y_1],[z_0:z_1]\big)&\mapsto\big([x_0:-x_1],[y_0:-y_1],[z_0:z_1]\big).
\end{align*}
Then $A\simeq(\mathbb{Z}/2\mathbb{Z})^3$, and $A$ does not fix points in $X$.
\end{example}

\section{Proof of Proposition\,B}
\label{section:auxiliary}

Let $X$ be a smooth Fano 3-fold defined over a subfield $\Bbbk\subset\mathbb{C}$. 
If $X$ is contained in the family \textnumero 1.15, it follows from \cite[Theorem~1.1]{KuznetsovProkhorov2023} that $X(\Bbbk)\ne\varnothing$.
Similarly, if $X$ is contained in one of the families
\begin{center}
\textnumero 2.9, \textnumero 2.11, \textnumero 2.14, \textnumero 2.17, \textnumero 2.20, \textnumero 2.22, \textnumero 2.26, \textnumero 2.28, \textnumero 2.30,\\ 
\textnumero 2.31, \textnumero 2.35, \textnumero 2.36, \textnumero 3.8, \textnumero 3.11, \textnumero 3.14, \textnumero 3.16,  \textnumero 3.18, \textnumero 3.21, \\
\textnumero 3.22, \textnumero 3.23, \textnumero 3.24, \textnumero 3.26, \textnumero 3.29, \textnumero 3.30, \textnumero 4.5, \textnumero 4.9, \textnumero 4.11,
\end{center}
then it follows from \cite{AbbanCheltsovKishimotoMangolte-1} that $X(\Bbbk)\ne\varnothing$. Likewise, we prove the following result. 

\begin{lemma}
\label{lemma:P}
If $X$ is contained in one of the families \textnumero 1.11, \textnumero 2.1, \textnumero 2.15, \textnumero 3.15, \textnumero 5.1,
then  $X(\Bbbk)\ne\varnothing$.
\end{lemma}

\begin{proof}
If $X$ is contained in the family \textnumero 1.11, then $-K_{X_{\mathbb{C}}}\sim 2H$ for an ample divisor $H\in\mathrm{Pic}(X_{\mathbb{C}})$ such that $H^3=1$,
and the base locus of the linear system $|H|$ consists of a single point, which must be defined over $\Bbbk$.
so, in particular, $X(\Bbbk)\ne\varnothing$.
Similarly, if $X$ is contained in the deformation family \textnumero 2.1,
then it follows from \cite[Lemma 2.5]{AbbanCheltsovKishimotoMangolte-1} that there exists birational morphism $X\to Y$
such that $Y$ is a smooth member of the deformation family \textnumero 1.11, so it follows from Lang-Nishimura theorem that $X(\Bbbk)\ne\varnothing$,
because we just proved that $Y(\Bbbk)\ne\varnothing$.

If $X$ is contained in the deformation family \textnumero 3.15,
then $X_{\mathbb{C}}$ can be realized as a blow up of a smooth quadric $Q\subset\mathbb{P}^4$ along a disjoint union of a line $\ell$ and a smooth conic $C$,
and it follows from \cite{Matsuki} and \cite[Corollary 2.3]{AbbanCheltsovKishimotoMangolte-1} that this blow up, the quadric $Q$, the line $\ell$ and the conic $C$ are all defined over $\Bbbk$, so, in particular, $Q( \Bbbk) \neq \emptyset$, which implies $X(\Bbbk) \neq\emptyset$ by Lang-Nishimura theorem \cite[Theorem 3.6.11]{Poonen}.

Now, we assume that $X$ is contained in the family \textnumero 2.15. 
Then it follows from \cite[Lemma 2.5]{AbbanCheltsovKishimotoMangolte-1} that there exists a birational morphism $\pi\colon X\to U$ such that
$U$ is a  $\Bbbk$-form of $\mathbb{P}^3$, and $\pi$ is a blow up of a smooth curve $C\subset U$ such that $C_\mathbb{C}$ is a complete intersection in $U_{\mathbb{C}}\simeq\mathbb{P}^3$ of a quadric surface $S_2$ and a cubic surface $S_3$.
Since $S_2$ is the unique quadric surface that contains $C_{\mathbb{C}}$, we see that $S_2$ is defined over $\Bbbk$.
Let $\mathcal{M}$ be the linear subsystem in $|S_2|$ that consists of all surfaces containing $C$.
Then $\mathcal{M}$ gives a birational map $U\dasharrow Y$ such that $Y_\mathbb{C}$ is a cubic 3-fold in $\mathbb{P}^4$ that has one isolated double point.
Now, applying \cite[Corollary 2.3]{AbbanCheltsovKishimotoMangolte-1}, we see that $Y$ is a cubic 3-fold in $\mathbb{P}^4$,
so projecting from its singular point, we obtain a birational map $Y\dasharrow\mathbb{P}^3$, which implies that $X(\Bbbk)\ne\varnothing$ by Lang-Nishimura theorem. 

Finally, we suppose that $X$ belongs to family \textnumero 5.1. 
Then, using  \cite[\S~III.3]{Matsuki} and \cite[Corollary~2.3]{AbbanCheltsovKishimotoMangolte-1},
we see that there is a birational morphism $f\colon X\to Q$ 
such that $Q$ is a smooth quadric in $\mathbb{P}^4$,  
and $f$ is a composition of a blow up of a reduced zero-dimensional subscheme $\Sigma\subset Q$ of length $3$ followed by the blow up of a strict transform of a conic in $Q$ that contains $\Sigma$. Then $Q(\Bbbk)\ne\varnothing$ by Springer theorem \cite{Springer}, so Lang-Nishimura theorem gives  $X(\Bbbk)\ne\varnothing$. 
\end{proof}

Over $\mathbb{C}$, the deformation families
\begin{center}
\textnumero 2.34, \textnumero 3.27, \textnumero 3.28, \textnumero 4.10, \textnumero 5.3, \textnumero 6.1, \textnumero 7.1, \textnumero 8.1, \textnumero 9.1, \textnumero 10.1
\end{center}
consist of products $\mathbb{P}^1\times S$, where $S$ is a smooth del Pezzo surface.
Thus, each of these families contains a real pointless 3-fold $C_2\times S$ such that $C_2$ is a pointless real conic in $\mathbb{P}^2_{\mathbb{R}}$,
and $S$ is a real smooth del Pezzo surface of an appropriate anticanonical degree.
Moreover, it has been shown in \cite{AbbanCheltsovKishimotoMangolte-1} that families
\begin{center}
\textnumero 1.10, \textnumero 2.10, \textnumero 2.12, \textnumero 2.13, \textnumero 2.16, \textnumero 2.19, \textnumero 2.21, \textnumero 2.23, \textnumero 2.33,\\
\textnumero 3.2, \textnumero 3.5, \textnumero 3.6, \textnumero 3.10, \textnumero 3.12, \textnumero 3.13, \textnumero 3.31, \textnumero 4.8, \textnumero 4.12, \textnumero 4.13,  \textnumero 5.2
\end{center}
contain real smooth pointless Fano 3-folds, and families \textnumero 1.9,  \textnumero 2.5,  \textnumero 2.24, \textnumero 3.7
contain smooth Fano 3-fold defined over $\mathbb{Q}$ that do not have rational points.
Furthermore, 
Examples~\ref{example:A-1-1-1-12}, \ref{example:A-1-2}, \ref{example:A-1-3}, \ref{example:A-1-4}, \ref{example:A-2-3}, \ref{example:A-2-7}, \ref{example:A-2-8}, \ref{example:A-3-3}, \ref{example:A-3-9} contains explicit examples of real smooth pointless Fano 3-folds in the families 
\begin{center}
\textnumero 1.1, \textnumero 1.2, \textnumero 1.3,  \textnumero 1.4, \textnumero 1.12, \textnumero 1.16, \textnumero 2.3, \textnumero 2.7, \textnumero 2.8, \textnumero 2.29, \textnumero 3.3, \textnumero 3.9, \textnumero 3.19, \textnumero 3.20, \textnumero 4.4.
\end{center}
Similarly, Example~\ref{example:A-1-14} gives examples of real smooth pointless Fano 3-folds in the family \textnumero 1.14
if we choose $a_1,a_2,a_3,a_4,a_5,a_6$ in this example to be real,
and Example~\ref{example:A-4-2} gives many  examples of real smooth pointless Fano 3-folds in the family \textnumero 4.2 if we choose $a_1,a_2,a_3,a_4$ there to be real.
Hence, to prove Proposition\,B, it is enough to present an example of a smooth Fano 3-fold $X$ in each of the families 
\begin{center}
\textnumero 1.5, \textnumero 1.6, \textnumero 1.7, \textnumero 1.8, \textnumero 1.13, \textnumero 1.17, \textnumero 2.2, \textnumero 2.4, \textnumero 2.6, \textnumero 2.18, \textnumero 2.25, \\
\textnumero 2.27, \textnumero 2.32,  \textnumero 3.1, \textnumero 3.4,  \textnumero 3.9, \textnumero 3.17,  \textnumero 3.25, \textnumero 4.1,  \textnumero 4.3, \textnumero 4.6, \textnumero 4.7
\end{center}
such that $X$ is defined over an appropriate subfield $\Bbbk\subset\mathbb{C}$ and $X(\Bbbk)=\varnothing$.
We will do this in the remaining part of the section indicating whether $\Bbbk=\mathbb{R}$ or $\Bbbk=\mathbb{Q}$ or $\Bbbk$ is some other field.
 
\begin{example}
\label{example:P-1-5}
Let $V$ be the real Grassmannian $\mathrm{Gr}(2,5)$ embedded into $\mathbb{P}^{9}$ by the Pl\"ucker embedding,
let $H_1$ and $H_2$ be general hyperplanes in $\mathbb{P}^9$, let $Q$ be a general pointless quadric in $\mathbb{P}^9$,
and let $X$ be the intersection of $V$, $H_1$, $H_2$, $Q$. Then $X$ is a real pointless smooth Fano 3-fold in the family \textnumero 1.5.
\end{example}

\begin{example}
\label{example:P-1-6-1-7-1-8}
By Corollary\,A, each family among \textnumero 1.6, \textnumero 1.7, \textnumero 1.8 
contains a smooth 3-fold $X$ defined over some subfield $\Bbbk\subset\mathbb{C}$ such that $X$ is not $\Bbbk$-unirational,
so $X(\Bbbk)=\varnothing$ by \cite[Theorem~1.1]{KuznetsovProkhorov2023}.
\end{example}

\begin{example}
\label{example:P-1-13}
To construct smooth pointless member of the family \textnumero 1.13, let
$$
X=\big\{x_1^3+2x_2^3+4x_3^3+x_1x_2x_3+7(x_4^3+2x_5^3)=0\big\}\subset\mathbb{P}^4_{x_1,x_2,x_3,x_4,x_5}.
$$
Then $X$ is a smooth cubic 3-fold defined over $\mathbb{Q}$, and it follows from \cite{DaiXu} that $X(\mathbb{Q})=\varnothing$.
\end{example}

\begin{example}
\label{example:P-1-17}
Let $U$ be the unique real form of $\mathbb{P}^3$ that has no real points.
Then $U$ is a smooth Fano 3-fold in the family \textnumero 1.17,
and $U$ contains a smooth surface $S$ such that $-K_U\sim 2S$ and $\mathrm{Pic}(U)=\mathbb{Z}[S]$.
Moreover, it follows from \cite[Example~6.8]{AbbanCheltsovKishimotoMangolte-1} that $S\simeq \mathbb{P}^1\times C$,
where $C$ is the real pointless conic.
This shows that $S$ contains three twisted lines $L$, $L^\prime$, $L^{\prime\prime}$, 
that is, $L_{\mathbb{C}}$, $L_{\mathbb{C}}^\prime$, $L_{\mathbb{C}}^{\prime\prime}$ are disjoint lines in $U_{\mathbb{C}}\simeq\mathbb{P}^3$.
If we blow up $U$ along $L$ and $L^\prime$, we obtain smooth real pointless Fano 3-fold in the family \textnumero 3.25.
Similarly, if we blow up $U$ along $L$, $L^\prime$, $L^{\prime\prime}$, we obtain smooth real pointless Fano 3-fold in the family \textnumero 4.6.
Moreover, if $S^\prime$ is a general surface in $|S|$ different from $S$, then $S\cap S^\prime$ is a smooth elliptic curve,
so blowing up $U$ along this curve, we obtain smooth real pointless Fano 3-fold in the family \textnumero 2.25. 
Finally, let $Z$ be a general curve in $S$ that is contained in the linear system $|-K_S-L|$.
Then $Z_{\mathbb{C}}$ is a smooth twisted rational cubic in  $U_{\mathbb{C}}\simeq\mathbb{P}^3$,
so blowing up $U$ along $Z$,  we get pointless real Fano 3-fold in the family \textnumero 2.27. 
\end{example}

\begin{example}
\label{example:P-2-2-2-18}
Let us use assumptions and notations of Example~\ref{example:A-2-2-2-18}.
Then $S$ is real and $S(\mathbb{R})=\varnothing$.
Thus, we can choose $X$ to be real and pointless.
Recall that $X$ is contained in the family \textnumero 2.18 if $d=2$, 
and $X$ is contained in the family \textnumero 2.2 if $d=4$.
\end{example}

\begin{example}
\label{example:P-2-4}
Let $C$ be the pointless real conic, let $U$ be the pointless real form of $\mathbb{P}^3$, and let $V=C\times U$.
Then, in the notations of \cite[\S~4]{KollarSB}, the twisted line bundle $\mathcal{O}_V(1,1)$ is a line bundle on $V$ by \cite[\S~4]{KollarSB}.
Thus, since $\mathcal{O}_U(2)$ is a line bundle on $U$, we see that $\mathcal{O}_V(1,3)$ is a line bundle on $V$.
Let $X$ be a general 3-fold in $|\mathcal{O}_V(1,3)|$. Then $X$ is smooth, and 
$X_\mathbb{C}\sim \mathrm{pr}_1^*(\mathcal{O}_{\mathbb{P}^1}(1))+\mathrm{pr}_2^*(\mathcal{O}_{\mathbb{P}^3}(3))$  on $V_\mathbb{C}\simeq\mathbb{P}^1\times\mathbb{P}^3$, where $\mathrm{pr}_1\colon \mathbb{P}^1\times\mathbb{P}^3\to \mathbb{P}^1$ and $\mathrm{pr}_2\colon \mathbb{P}^1\times\mathbb{P}^3\to \mathbb{P}^3$ are projections to the first and the second factor, respectively.
Hence, since $V(\mathbb{R})=\varnothing$, $X$ is a pointless real Fano 3-fold in the family \textnumero 2.4.
\end{example}

\begin{example}
\label{example:P-2-6}
Let $X$ be the divisor of degree $(2,2)$ in $\mathbb{P}^2_{u,v,w}\times\mathbb{P}^2_{x,y,z}$ that is given by 
$$
(u^2+v^2+1967w^2)x^2+(u^2+1973v^2+w^2)y^2+(1983u^2+v^2+w^2)z^2=0.
$$
Then $X$ is a real smooth pointless Fano 3-fold in the family \textnumero 2.6.
\end{example}

\begin{example}
\label{example:P-2-32}
Let us use the notation and assumptions of Example~\ref{example:A-2-7} with    
$$
Z=\{x_1=x_2+ix_3=x_4+ix_5=0\}\cup\{x_1=x_2-ix_3=x_4-ix_5=0\}. 
$$ 
Let $S=Q\cap \{x_1=0\}$, and let $\widetilde{S}$ be the strict transform of $S$ on the 3-fold $X$.
Then there is a birational morphism $X\to W$ such that $W$ is a smooth Fano 3-fold in the family \textnumero 2.32.
By \cite[Theorem 3.6.11]{Poonen}, we have $W(\mathbb{R})=\varnothing$, since $X(\mathbb{R})=\varnothing$.
Now, we let 
$$
C=\{x_3=x_2+ix_1=x_4+ix_5=0\}\cup\{x_3=x_2-ix_1=x_4-ix_5=0\}.
$$ 
Then $C\subset Q$, and the curve $C$ is defined over $\mathbb{R}$.
Let $\widetilde{C}$ be the strict transform on $X$ of the curve $C$. Then $\widetilde{C}\not\subset\widetilde{S}$, and the image of $\widetilde{C}$ in $W$ is a smooth curve. 
Moreover, if we blowup $W$ along this curve, we obtain a real pointless smooth Fano 3-fold in the family \textnumero 4.7.
\end{example}

\begin{example}
\label{example:P-3-1}
Let $Q$ be a pointless real quadric in $\mathbb{P}^3$, let $V=Q\times \mathbb{P}^1$, let $S$ be a general surface in the~linear system $|-K_V|$,
and let $\pi\colon X\to V$ be a double cover branched over $S$. Then $X$ is a smooth real pointless Fano 3-fold in the family \textnumero 3.1. 
\end{example}

\begin{example}
\label{example:P-3-4}
Let us use the notation and assumptions of Example~\ref{example:P-2-2-2-18} with $d=2$.
Let $C$ be the preimage via $\pi$ of the curve $\{y=0,z=0\}\subset \mathbb{P}^1_{u,v}\times\mathbb{P}^2_{x,y,z}$.
Then $C$ is smooth and defined over $\mathbb{R}$. Let $Y$ be the blow up of the 3-fold $X$ along the curve $C$.
Then $Y$ is a smooth real Fano 3-fold in the family \textnumero 3.4, which is pointless by \cite[Theorem 3.6.11]{Poonen}.
\end{example}

\begin{example}
\label{example:P-3-9}
Let us use the notation and assumptions of Example~\ref{example:A-3-9}. 
Then $S$ and $E+E^\prime$ are defined over $\mathbb{R}$, and the divisor $S+E+E^\prime$ does not contain real points.
Thus, we can choose $\overline{X}$ to be real and pointless. 
Then, by construction, $X$ is a smooth real pointless Fano 3-fold in the family \textnumero 3.9.
\end{example}

\begin{example}
\label{example:P-3-17}
Let $V=Q\times\mathbb{P}^2$, where $Q$ is a pointless real quadric in $\mathbb{P}^3$,
let $X$ be a general divisor in $|\mathrm{pr}_1^*(H)+\mathrm{pr}_2^*(\mathcal{O}_{\mathbb{P}^2}(1))|$,
where $H$ is a hyperplane section of $Q$, $\mathrm{pr}_1\colon V\to Q$ and $\mathrm{pr}_2\colon V\to \mathbb{P}^2$ are projections to the first and the second factors, respectively.  Then $X$ is a smooth real pointless Fano 3-fold in the family \textnumero 3.17.
\end{example}

\begin{example}
\label{example:P-4-1}
Let $X$ be the 3-fold in $\mathbb{P}^3_{x_1,x_2,x_3,x_4}\times\mathbb{P}^{3}_{y_1,y_2,y_3,y_4}$ given by
$$
\left\{\aligned
&x_1^2+x_2^2+x_3^2+x_4^2=0, \\
&y_1^2+y_2^2+y_3^2+y_4^2=0,\\
&x_1y_1+1967x_2y_2+1973x_3y_3+1983x_4y_4=0.
\endaligned
\right.
$$
Then $X$ is a smooth real pointless Fano 3-fold in the deformation family \textnumero 4.1.
\end{example}

\begin{example}
\label{example:P-4-3}
Let $V=\{x_1^2+x_2^2+x_3^2=y_1^2+y_2^2+y_3^2=0\}\subset\mathbb{P}^2_{x_1,x_2,x_3}\times\mathbb{P}^{2}_{y_1,y_2,y_3}\times\mathbb{P}^{1}_{z_1,z_2}$,
and let
$$
C=V\cap\big\{x_1y_2=x_2y_1, x_1y_3=x_3y_1, x_2y_3=x_3y_2,x_1z_2=x_2z_1\big\}.
$$
Then $V$ and $C$ are smooth. Let $X$ be the blowup of $V$ along the curve $C$. 
Then $X$ is a pointless smooth real Fano 3-fold in the family \textnumero 4.3.
\end{example}

\appendix

\section{Unirationality of smooth Fano 3-folds of degree $14$}
\label{section:V14}


Let $X$ be a smooth Fano 3-fold contained in the family \textnumero 1.7 which is defined over a subfield $\Bbbk\subset\mathbb{C}$.
Then $\mathrm{Pic}(X_{\mathbb{C}})=\mathbb{Z}[-K_{X_{\mathbb{C}}}]$, $-K_X^3=14$, the linear system $|-K_X|$ gives an embedding $X\hookrightarrow\mathbb{P}^{9}$ 
such that the image of $X$ is a scheme-theoretic intersection of quadrics.
In the following, we will identify $X$ with its anticanonical image in $\mathbb{P}^9$.
The goal of this appendix is to present a short proof of the following result, which essentially follows from the ideas and results in \cite{Kollar,KuznetsovProkhorov2023,Takeuchi}.

\begin{theorem}
\label{theorem:V14}
The following three conditions are equivalent:
\begin{itemize}
\item[$(\mathrm{1})$] $X(\Bbbk)\ne\varnothing$,
\item[$(\mathrm{2})$] $X(\Bbbk)\ne\varnothing$ and $X$ is birational to a smooth cubic 3-fold in $\mathbb{P}^4$,
\item[$(\mathrm{3})$] $X$ is unirational over $\Bbbk$.
\end{itemize}
\end{theorem}

To prove this theorem, we need two results, which are known to experts.

\begin{lemma}
\label{lemma:V14-line}
Suppose that $X$ contains a line $\ell$ defined over $\Bbbk$. Then $X$ is unirational over $\Bbbk$.
\end{lemma}

\begin{proof}
Let $\pi\colon\widetilde{X}\to X$ be the blowup of the line $\ell$.
Then $|-K_{\widetilde{X}}|$ is base point free and, in particular, the divisor $-K_{\widetilde{X}}$ is big and nef.
Moreover, it follows from  \cite[Theorem~4.3.1]{IsPr99} that $|-K_{\widetilde{X}}|$ gives a small birational morphism $\phi\colon\widetilde{X}\to Y$ such that $Y$ is a singular Fano 3-fold with terminal Gorenstein singularities, and $(-K_Y)^3=10$. 
So, by \cite[Theorem~4.3.3]{IsPr99}, we have the following Sarkisov link:
$$
\xymatrix{
\widetilde{X}\ar@{->}[d]_{\pi}\ar@{->}[drr]_{\phi}\ar@{-->}[rrrr]^{\chi}&&&& \widetilde{X}'\ar@{->}[dll]^{\phi^\prime}\ar@{->}[d]^{\eta}\\%
X && Y&&S}
$$
where $\chi$ is a pseudoisomorphism that flops curves contracted by $\phi$, 
$\widetilde{X}'$ is a smooth weak Fano 3-fold, $\phi^\prime$ is a small birational morphism,
$S$ is a form of $\mathbb{P}^2$, and $\eta$ is a conic bundle.
Since $\widetilde{X}'(\Bbbk)\ne\varnothing$ by Lang–Nishimura theorem \cite[Theorem 3.6.11]{Poonen}, we have $S(\Bbbk)\ne\varnothing$, so $S\simeq\mathbb{P}^2$.

Let $E$ be the $\pi$-exceptional divisor. Then $E$ is rational over $\Bbbk$,
and it follows from \cite[Theorem~4.3.3]{IsPr99} that $\eta$ induces a dominant morphism $\chi_\ast(E)\to S$,
so $\widetilde{X}'$ is unirational over $\Bbbk$ by  \cite[Lemma~4.14]{KuznetsovProkhorov2023}.
Therefore, the smooth Fano 3-fold $X$ is also unirational over $\Bbbk$.
\end{proof}

\begin{lemma}
\label{lemma:V14-point}
Suppose that $X(\Bbbk)$ contains a point $P$ such that the 3-fold $X_{\mathbb{C}}$ does not contain lines passing through~$P$. 
Then $X$ is birational to a smooth cubic 3-fold in $\mathbb{P}^4$.
\end{lemma}

\begin{proof}
Let $\pi\colon \widetilde{X}\to X$ be the blowup of the point $P$.
Then $-K_{\widetilde{X}}^3=6$, and it follows from \cite[Lemma~5.7]{KuznetsovProkhorov2023} that $|-K_{\widetilde{X}}|$ is a base point free linear system of dimension $5$. Thus, it follows from \cite{CPS} and the proof of \cite[Lemma~5.11]{KuznetsovProkhorov2023} that the linear system $|-K_{\widetilde{X}}|$ gives a morphism $\varphi\colon\widetilde{X}\to\overline{X}$ such that either $\varphi$ is birational, 
or $\varphi$ is generically two-to-one, and $\overline{X}$ is either a Segre cubic scroll or a cone over a smooth two-dimensional cubic scroll.
Moreover, arguing as in the proof of \cite[Lemma~5.12]{KuznetsovProkhorov2023}, we see that $\varphi$ is birational,
and $\overline{X}$ is a normal complete intersection of a quadric and a cubic in $\mathbb{P}^5$.

Let $E$ be the $\pi$-exceptional surface.
Then the dimension of $|-2K_{\widetilde{X}}-E|$ is at least $5$ by  \cite[Lemma~5.4]{KuznetsovProkhorov2023}.
On the other hand, if $\varphi$ contracts an irreducible surface $S$, then it follows from \cite[Lemma~5.7]{KuznetsovProkhorov2023} that 
$$
S\sim t(-2K_{\widetilde{X}}-3E)
$$
for some $t\in\mathbb{Z}_{>0}$, and $S$ is a fixed component of the linear system  $|-2K_{\widetilde{X}}-E|$,
which is a contradiction. Thus, we see that $\varphi$ is small, which also follows from \cite[Theorem~4.9]{JahnkePeternellRadloff2005}.
Therefore, $\overline{X}$ has terminal Gorenstein singularities. 
Then the blowup $\pi$ gives rise to a Sarkisov link studied in \cite{Takeuchi}.
To be more precise, it follows from \cite{Takeuchi} or \cite[Theorem~4.5.8]{IsPr99} that 
the linear system $|-2K_{\widetilde{X}}-E|$ gives a birational map $\rho\colon\widetilde{X}\dasharrow V$ such that $V$ is a smooth cubic 3-fold in $\mathbb{P}^4$,
and $\rho$ fits the following commutative diagram:
$$
\xymatrix{
\widetilde{X}\ar@{->}[d]_{\pi}\ar@{-->}[drr]_{\rho}\ar@{-->}[rr]^{\chi}&& \widetilde{X}'\ar@{->}[d]^{\eta}\\%
X && V}
$$
where $\chi$ is a pseudoisomorphism that flops curves contracted by $\varphi$, $\widetilde{X}'$ is a smooth weak Fano 3-fold, 
and $\eta$ is a blowup of a form of a twisted rational quartic curve in $V$. 
\end{proof}

Now, we are ready to prove Theorem~\ref{theorem:V14}.
If $X$ is unirational over $\Bbbk$, then $X(\Bbbk)$ is Zariski dense in $X$,
which implies that $X(\Bbbk)$ contains a point $P$ such that $X_{\mathbb{C}}$ does not contain lines that pass through $P$,
so $X$ is birational to a smooth cubic 3-fold in $\mathbb{P}^4$ by Lemma~\ref{lemma:V14-point}.
This proves $(\mathrm{3})\Rightarrow (\mathrm{2})$.

If $X(\Bbbk)\ne\varnothing$ and $X$ is birational to a smooth cubic 3-fold $Y\subset\mathbb{P}^4$, then $Y(\Bbbk)\ne\varnothing$ by Lang–Nishimura theorem \cite[Theorem~3.6.11]{Poonen}, which implies that $Y$ is  unirational over~$\Bbbk$ by \cite{Kollar}, so $X$ is unirational over~$\Bbbk$ as well.
This proves $(\mathrm{2})\Rightarrow (\mathrm{3})$. 

The implication $(\mathrm{2})\Rightarrow (\mathrm{1})$ is obvious, hence to complete the proof, we show that $(\mathrm{1})\Rightarrow (\mathrm{3})$. 
To do this, we suppose that $X$ contains a $\Bbbk$-point $P$. We must prove that $X$ is unirational over $\Bbbk$. 
Using Lemmas~\ref{lemma:V14-line} and \ref{lemma:V14-point}, we may assume that $P$ is contained in a line in~$X_{\mathbb{C}}$, 
but none of the lines in $X_{\mathbb{C}}$ that passes through $P$ is defined over $\Bbbk$. 
In particular, the 3-fold $X_{\mathbb{C}}$ contains at least two lines that pass through $P$.

As in \cite[\S~2.4]{KuznetsovProkhorov2023}, we let $F_1(X)$ be the Hilbert scheme of lines in~$X$,
and we let $F_1(X,P)$ be the subscheme in $F_1(X)$ parameterizing lines passing through $P$. Then 
$$
F_1(X_{\mathbb{C}},P)\simeq F_1(X,P)_{\mathbb{C}}.
$$
Moreover, by our assumption, we have $F_1(X,P)(\Bbbk)=\varnothing$ and $F_1(X,P)(\mathbb{C})\ne\varnothing$.
Furthermore, it follows from \cite[Corollary~A.6]{KuznetsovProkhorov2023} that the length of the subscheme $F_1(X,P)$ is at most $3$,
so this subscheme must be reduced, since otherwise we would have $F_1(X,P)(\Bbbk)\ne\varnothing$. Set 
$$
C=\mathbf{T}_P(X)\cap X
$$
where $\mathbf{T}_P(X)$ stands for the embedded tangent space in $\mathbb{P}^9$ to the 3-fold $X$ at the point $P$. 
Then it follows from \cite[Lemma~5.6]{KuznetsovProkhorov2023} that $C$ is the cone over $F_1(X,P)$. 
Therefore, we see that $C$ is a reduced, irreducible, geometrically reducible curve, and one of the following two possibilities holds:
\begin{enumerate}
\item either $C$ is a conic, and $C_{\mathbb{C}}$ is a union two lines that intersect at $P$,
\item or $C$ is a cubic, and $C_{\mathbb{C}}$ is a union of three non-coplanar lines that intersect at $P$.
\end{enumerate}

In both cases, let $\pi\colon \widetilde{X}\to X$ be the blowup of the point $P$,
let $\widetilde{C}$ be the strict transform on $\widetilde{X}$ of the curve $C$,
let $\sigma\colon\widehat{X}\to\widetilde{X}$ be the blowup of the curve $\widetilde{C}$,
let $E$ be the $\pi$-exceptional surface, let $F$ be the $\sigma$-exceptional surface, 
and let $\widehat{E}$ be the strict transform on $\widehat{X}$ of the surface $E$.
Then $-K_{\widehat{X}}^3=6$ and 
$$
-K_{\widehat{X}}\sim (\pi\circ\sigma)^*\big(-K_X\big)-2\widehat{E}-F.
$$
Moreover, arguing as in the proofs of \cite[Lemma~5.11]{KuznetsovProkhorov2023} and \cite[Lemma~5.12]{KuznetsovProkhorov2023},
we see that the linear system $|-K_{\widehat{X}}|$ is base points free, and it gives a birational morphism  
$\varphi\colon\widehat{X}\to\overline{X}$ such that $\overline{X}$ is a normal complete intersection of a quadric hypersurface $Q\subset\mathbb{P}^5$
and a cubic hypersurface in $\mathbb{P}^5$.
Hence, we see that $\overline{X}$ is a Fano 3-fold that has canonical Gorenstein singularities.

\begin{lemma}
\label{lemma:V14-1}
The birational morphism $\varphi$ is small.
\end{lemma}

\begin{proof}
The required assertion follows from the proof of \cite[Lemma~5.14]{KuznetsovProkhorov2023}.
\end{proof}

In particular, the 3-fold $\overline{X}$ has (isolated) terminal Gorenstein singularities.
Then, arguing as in the proof of \cite[Lemma~5.14]{KuznetsovProkhorov2023}, we see that $\overline{X}$ is not a cone and $\overline{X}$ is not covered by lines.

\begin{lemma}
\label{lemma:V14-2}
The quadric $Q$ has at most one singular point.
\end{lemma}

\begin{proof}
Suppose that $\mathrm{Sing}(Q)$ is a line $\ell$. Then the hyperplane class of $Q\setminus\ell$ can be
represented as the sum of two movable classes, so the same is true for $\overline{X}\setminus\mathrm{Sing}(X)$, 
hence the same is true for $-K_{\widehat{X}}$ and therefore for the hyperplane class of $X$, which is absurd,
since $\mathrm{Pic}(X)$ is generated by $-K_X$. Similarly, we see that $\mathrm{Sing}(Q)$ cannot be a plane.
\end{proof}

Set $\overline{E}=\varphi(\widehat{E})$. Then $\overline{E}$ is a $\Bbbk$-rational surface in $\overline{X}\subset\mathbb{P}^5$ such that
$$
\mathrm{deg}(\overline{E})=(-K_{\widehat{X}})^2\cdot\widehat{E}=\left\{\aligned
&2\ \text{if $C$ is a reduced conic},\\ 
&1\ \text{if $C$ is a reduced cubic}.
\endaligned
\right.
$$
Hence, if $C$ is a reduced cubic, then $\overline{E}$ is a plane.
Similarly, if $C$ is a reduced conic, then $\overline{E}$ is a quadric surface.
Now, the unirationality of $X$ follows from the following two lemmas.

\begin{lemma}
\label{lemma:V14-4}
Suppose that $C$ is a conic. Then $X$ is unirational over $\Bbbk$.
\end{lemma}

\begin{proof}
By construction, $\overline{E}$ is a smooth quadric surface in $\mathbb{P}^5$.
In fact, we have $\overline{E}\simeq\mathbb{P}^1\times\mathbb{P}^1$. 
Let $\Pi$ be the three-dimensional linear subspace in $\mathbb{P}^5$ that contains the quadric $\overline{E}$. 
Then $\Pi\not\subset Q$, because $Q$ has at most one singular point. 
This gives $\overline{E}=Q\cap\Pi=\overline{X}\cap\Pi$.

Let $\psi\colon V\dasharrow\mathbb{P}^1$ be the rational map given by the projection from $\Pi$.
Then $\psi$ is given by the linear system $|-K_{\overline{X}}-\overline{E}|$, and we have the following commutative diagram:
$$
\xymatrix@R=1em{
&Y\ar@{->}[dl]_{\rho}\ar@{->}[dr]^{\eta}&\\%
\overline{X}\ar@{-->}[rr]_{\psi} && \mathbb{P}^1}
$$
where $\rho$ is a birational morphism induced by the blowup of $\mathbb{P}^5$ along $\Pi$, and $\eta$ is a morphism. 
Moreover, the birational morphism $\rho$ is a small, because $\overline{X}$ has terminal Gorenstein singularities.
Thus, we see that $-K_{Y}\sim\rho^*(-K_{\overline{X}})$, and $Y$ also has terminal Gorenstein singularities. 

Let $E_Y$ be the strict transform of the surface $E$ on the threefold $Y$, let $S$ be a general fiber of $\eta$,
and let $\overline{S}=\rho(S)$.
Then $S\sim -K_{Y}-E_Y$, and 
$$
\mathrm{deg}(\overline{S})=(-K_Y)^2\cdot S=(-K_Y)^2\cdot (-K_{Y}-E_Y)=(-K_{\widehat{X}})^2\cdot (-K_{\widehat{X}}-\widehat{E})=4,
$$
which implies that $\overline{S}$ is an irreducible surface of degree $4$, because $\overline{X}$ is not covered by lines.
Now, it follows from the adjunction formula that $S$ is a smooth del Pezzo surface of degree $(-K_S)^2=(-K_Y)^2\cdot S=4$.
Furthermore, we have 
$$
-K_{Y}\cdot E_Y\cdot (-K_{Y}-E_Y)=(-K_{\widehat{X}})\cdot \widehat{E}\cdot (-K_{\widehat{X}}-\widehat{E})=4\ne 0,
$$
which implies that the restriction morphism $\eta\vert_{E_Y}\colon E_Y\to\mathbb{P}^1$ is surjective.
Then $Y$ is unirational over $\Bbbk$ by \cite[Lemma~4.14]{KuznetsovProkhorov2023}, because  $E_Y$ is rational over $\Bbbk$.
Hence, $X$ is also unirational.
\end{proof}

\begin{lemma}
\label{lemma:V14-3}
Suppose that $C$ is a cubic curve. Then $X$ is unirational over $\Bbbk$.
\end{lemma}

\begin{proof}
In this scenario, $\overline{E}$ is a plane.
Let $\psi\colon \overline{X}\dasharrow\mathbb{P}^2$ be the map given by the projection from $\overline{E}$.
Then $\psi$ is given by the linear system $|-K_{\overline{X}}-\overline{E}|$, and we have the following commutative diagram:
$$
\xymatrix@R=1em{
&Y\ar@{->}[dl]_{\rho}\ar@{->}[dr]^{\eta}&\\%
\overline{X}\ar@{-->}[rr]_{\psi} && \mathbb{P}^2}
$$
where $\rho$ is a birational morphism induced by the blowup of $\mathbb{P}^5$ along the plane $\overline{E}$, and $\eta$ is a morphism. 
Moreover, the birational morphism $\rho$ is small, because $\overline{X}$ has terminal Gorenstein singularities.
Thus, we see that $-K_{Y}\sim\rho^*(-K_{\overline{X}})$, and $Y$ also has terminal Gorenstein singularities. 

Let $E_Y$ be the strict transform of the surface $E$ on the threefold $Y$.
Then the morphism $\eta$ is given by the linear system $|-K_{Y}-E_Y|$, and 
$$
-K_Y\cdot (-K_{Y}-E_Y)^2=-K_{\widehat{X}}\cdot (-K_{\widehat{X}}-\widehat{E})^2=2.
$$
This implies that $\eta$ is surjective, and its general fiber is irreducible and isomorphic to $\mathbb{P}^1$, 
because otherwise the image in $\overline{X}$ of a general fiber of $\eta$ would be a union of two lines, but $\overline{X}$ is not covered by lines.
Hence, we see that $\eta$ is a conic bundle. Similarly, we see that $E_Y^3=-3$, because
$$
0=(-K_{Y}-E_Y)^3=(-K_{Y})^3-3(-K_{Y})^2\cdot E_Y+3(-K_{Y})\cdot E_Y-E_Y^3=-3-E_Y^3.
$$
This gives
$$
E_Y\cdot (-K_{Y}-E_Y)^2=E_Y\cdot (-K_{Y})^2-2(-K_{Y})\cdot E_Y^2+E^3=5+E^3=2,
$$
so the restriction morphism $\eta\vert_{E_Y}\colon E_Y\to\mathbb{P}^2$ is generically two-to-one. 
On the other hand, $E_Y$ is rational over $\Bbbk$ by construction.
Then $Y$ is unirational by \cite[Lemma~4.14]{KuznetsovProkhorov2023}, so $X$ is also unirational.
\end{proof}


\begin{thebibliography}{33}

\bibitem{one-dim}
H.~Abban, I.~Cheltsov, E.~Denisova, E.~Etxabarri-Alberdi, A.-S.~Kaloghiros, D.~Jiao, J.~Martinez-Garcia, T.~Papazachariou,
\textit{One-dimensional components in the K-moduli of smooth Fano 3-folds}, J.\ Algebraic Geom.\ \textbf{34} (2025), 489--534.

\bibitem{AbbanCheltsovKishimotoMangolte-1}
H.~Abban, I.~Cheltsov, T.~Kishimoto, F.~Mangolte, \emph{K-stability of pointless del Pezzo surfaces and Fano 3-folds}, preprint, arXiv:2411.00767, 2024.

\bibitem{AbbanCheltsovKishimotoMangolte-2}
H.~Abban, I.~Cheltsov, T.~Kishimoto, F.~Mangolte, \emph{K-stability of Fano 3-folds in the~World of Null-A}, preprint, arXiv:2505.04330, 2025.

\bibitem{Book}
C.~Araujo, A.-M.~Castravet, I.~Cheltsov, K.~Fujita, A.-S.~Kaloghiros, J.~Martinez-Garcia, C.~Shramov, H.~S\"u\ss, N.~Viswanathan, \emph{The Calabi problem for Fano threefolds}, Cambridge University Press, \textbf{485} (2023).

\bibitem{Beauville}
A.~Beauville, \emph{Finite simple groups of small essential dimension}, Springer INdAM Series \textbf{8} (2014), 221--228.

\bibitem{fanography}
P.~Belmans, \emph{Fanography}, \texttt{https://fanography.info}, 2025.

\bibitem{CheltsovGuerreiroFujitaKrylovMartinez}
I.~Cheltsov, T.~Duarte Guerreiro, K.~Fujita, I.~Krylov, J.~Martinez-Garcia, \emph{K-stability of Casagrande-Druel varieties},
to appear in Journal fur die Reine und Angewandte Mathematik.

\bibitem{CheltsovFedorchukFujitaKaloghiros}
I.~Cheltsov, M.~Fedorchuk, K.~Fujita, A.-S.~Kaloghiros, \emph{K-moduli of pure states of four qubits}, preprint, arXiv:2412.19972, 2024.

\bibitem{CPS}
I.~Cheltsov, V.~Przyjalkowski, C.~Shramov, \emph{Hyperelliptic and trigonal Fano threefolds}, Izv.\ Math.\ \textbf{69} (2005), 365--421.

\bibitem{CheltsovShramov}
I.~Cheltsov, C.~Shramov, \emph{Cremona groups and the~icosahedron},  CRC Press, Boca Raton, FL, 2016.

\bibitem{CheltsovTschinkelZhang2025}
I.~Cheltsov, Yu.~Tschinkel, Zh.~Zhang, \emph{Equivariant unirationality of Fano threefolds}, preprint, arXiv:2502.19598, 2025.

\bibitem{DaiXu}
H.~Dai, B.~Xue, \emph{Rational points on cubic hypersurfaces that split off two forms},
Bull\. Lond.\ Math.\ Soc.\ \textbf{46} (2014), 169--184.

\bibitem{DuncanReichstein}
A.~Duncan, Z.~Reichstein, \emph{Versality of algebraic group actions and rational points on twisted varieties},
J.\ Algebraic Geom.\ \textbf{24} (2015), 499--530.

\bibitem{IlievRanestad}
A.~Iliev, K.~Ranestad, \emph{Geometry of the Lagrangian Grassmannian $\mathrm{LG}(3,6)$ with applications to Brill-Noether loci}, 
Mich.\ Math.\ J.\ \textbf{53} (2005), 383--417.

\bibitem{IsPr99}
V.~Iskovskikh, Yu.~Prokhorov, \emph{Fano varieties},  Encyclopaedia of Mathematical Sciences \textbf{47} (1999) Springer, Berlin.

\bibitem{JahnkePeternellRadloff2005}
P.~Jahnke, T.~Peternell, I.~Radloff, \emph{Threefolds with big and nef anticanonical bundles.\ I}, Math.\ Ann.\ \textbf{333} (2005), 569--631.

\bibitem{Kollar}
J.~Koll\'ar, \emph{Unirationality of cubic hypersurfaces}, J. Inst. Math. Jussieu \textbf{1} (2002), 467--476.

\bibitem{KollarSB}
J.~Koll\'ar, \emph{Severi-Brauer varieties; a geometric treatment}, preprint, arXiv:1606.04368, 2016.

\bibitem{KuznetsovProkhorov2023}
A.~Kuznetsov, Yu.~Prokhorov, \emph{Rationality of Fano threefolds over non-closed fields},
Am.\ J.\ Math.\ \textbf{145} (2023), 335--411.

\bibitem{KuznetsovProkhorov2024}
A.~Kuznetsov, Yu.~Prokhorov, \emph{Rationality over nonclosed fields of Fano threefolds with higher geometric Picard rank},
J.\ Inst.\ Math.\ Jussieu \textbf{23} (2024), 207--247. 

\bibitem{LGBT}
The LMFDB Collaboration, \emph{The L-functions and modular forms database}, \texttt{https://www.lmfdb.org}, 2025.

\bibitem{Matsuki}
K.~Matsuki, \emph{Weyl groups and birational transformations among minimal models}, Mem.\ Am.\ Math.\ Soc.\ \textbf{557} (1995), 133 pages.

\bibitem{MoriMukai}
S.~Mori, S.~Mukai, \textit{Classification of Fano threefolds with $\mathrm{B}_2\geqslant 2$}, Manuscr.\ Math.\ \textbf{36} (1981), 147–162; 
Erratum, Manuscr.\ Math.\ \textbf{110} (2003), 407.

\bibitem{Poonen}
B.~Poonen, \emph{Rational points on varieties}, Graduate Studies in Mathematics \textbf{186}, American Mathematical Society, 2017.

\bibitem{Prokhorov-Simple}
Yu.~Prokhorov, \emph{Simple finite subgroups of the~Cremona group of rank~$3$},  J.\ Algebr.\ Geom.\ \textbf{21} (2012), 563--600.

\bibitem{ReYou00}
Z.~Reichstein, B.~Youssin, \emph{Essential dimensions of algebraic groups and a resolution theorem for $G$-varieties},
with an appendix by J.~Koll\'ar and E.~Szabo, Canad.\ J.\ Math.\ \textbf{52} (2000), no.\ 5, 1018--1056.

\bibitem{Reid1972}
M.~Reid, \textit{The complete intersection of two or more quadrics}, Ph.D.\ Thesis, Trinity College, Cambridge, 1972.

\bibitem{Reid}
M.~Reid, \emph{Chapters on algebraic surfaces}, IAS/Park City Math.\ Ser.\ \textbf{3} (1997), 3--159.

\bibitem{Springer}
T.~Springer, \textit{Sur les formes quadratiques d'indice zero}, C.\ R.\ Acad.\ Sci.\ Paris \textbf{234} (1952), 1517--1519.

\bibitem{Takeuchi}
K.~Takeuchi, \emph{Some birational maps of Fano 3-folds}, Compos.\ Math.\ \textbf{71} (1989), 265--283.

\bibitem{TschinkelZhang}
Yu.~Tschinkel, Z.~Zhang, \emph{Stable equivariant birationalities of cubic and degree 14 Fano threefolds}, preprint, arXiv:2409.08392, 2024.

\end{thebibliography}
\end{document}